\documentclass{amsart}

\usepackage{amsmath}
\usepackage{amsthm}
\usepackage{amsfonts}
\usepackage{amssymb}
\usepackage{mathrsfs}
\usepackage{amscd}
\usepackage{url}

\newcommand{\re}{\mathbb{R}}

\newcommand{\N}{\mathbb{N}}

\newcommand{\lmd}{\lambda}

\newcommand{\eps}{\epsilon}

\newcommand{\dt}{\delta}

\def\af{\alpha}
\def\bt{\beta}

\def\rank{\mbox{rank}}

\newcommand{\sig}{\sigma}
\newcommand{\Sig}{\Sigma}

\newcommand{\reff}[1]{(\ref{#1})}

\newcommand{\prm}{\prime}

\newcommand{\mc}[1]{\mathcal{#1}}

\newcommand{\supp}[1]{\mbox{supp}(#1)}

\newcommand{\bdes}{\begin{description}}
\newcommand{\edes}{\end{description}}

\newcommand{\bal}{\begin{align}}
\newcommand{\eal}{\end{align}}

\newcommand{\bnum}{\begin{enumerate}}
\newcommand{\enum}{\end{enumerate}}

\newcommand{\bit}{\begin{itemize}}
\newcommand{\eit}{\end{itemize}}

\newcommand{\bea}{\begin{eqnarray}}
\newcommand{\eea}{\end{eqnarray}}
\newcommand{\be}{\begin{equation}}
\newcommand{\ee}{\end{equation}}

\newcommand{\baray}{\begin{array}}
\newcommand{\earay}{\end{array}}

\newcommand{\bsry}{\begin{subarray}}
\newcommand{\esry}{\end{subarray}}

\newcommand{\bca}{\begin{cases}}
\newcommand{\eca}{\end{cases}}

\newcommand{\bcen}{\begin{center}}
\newcommand{\ecen}{\end{center}}

\newcommand{\bbm}{\begin{bmatrix}}
\newcommand{\ebm}{\end{bmatrix}}

\newcommand{\bmx}{\begin{matrix}}
\newcommand{\emx}{\end{matrix}}

\newcommand{\bpm}{\begin{pmatrix}}
\newcommand{\epm}{\end{pmatrix}}

\newcommand{\btab}{\begin{tabular}}
\newcommand{\etab}{\end{tabular}}

\newtheorem{theorem}{Theorem}[section]
\newtheorem{pro}[theorem]{Proposition}

\newtheorem{lem}[theorem]{Lemma}

\theoremstyle{definition}
\newtheorem{exm}[theorem]{Example}

\newcommand{\df}[1]{{\it{#1}}{\index{#1}}}
\newcommand{\ind}[1]{{{#1}}{\index{#1}}}

\begin{document}

\title[A Semidefinite Approach for TKMP]
{A Semidefinite Approach for Truncated $K$-Moment Problems}


\author[J.~William Helton]{J. William Helton}
\address{Department of Mathematics, University of California,
9500 Gilman Drive, La Jolla, CA 92093.}
\email{helton@math.ucsd.edu}

\author[Jiawang Nie]{Jiawang Nie}
\address{Department of Mathematics, University of California,
9500 Gilman Drive, La Jolla, CA 92093.}
\email{njw@math.ucsd.edu}

\begin{abstract}
A truncated moment sequence (tms) in $n$ variables and of degree $d$ is
a finite sequence $y=(y_\af)$ indexed by nonnegative integer vectors
$\af:=(\af_1,\ldots,\af_n)$ such that $\af_1+\cdots+\af_n\leq d$.
Let $K\subseteq \re^n$ be a semialgebraic set.
The truncated $K$-moment problem (TKMP) is:
how to check if a tms $y$ admits a $K$-measure $\mu$
(a nonnegative Borel measure supported in $K$) such that 
$y_\af = \int_K x_1^{\af_1}\cdots x_n^{\af_n}d\mu$ for every $\af$?
This paper proposes a semidefinite programming (SDP) approach for solving TKMP.
When $K$ is compact, we get the following results:
whether a tms admits a $K$-measure or not
can be checked  via solving a sequence of SDP problems;
when $y$ admits no $K$-measure, a certificate for the nonexistence can be found;
when $y$ admits one, a representing measure for $y$
can be obtained from solving the SDP problems under some necessary and
some sufficient conditions.
Moreover, we also propose a practical SDP method for finding flat extensions,
which in our numerical experiments always
found a finitely atomic representing measure when it exists.
\end{abstract}



\maketitle

\section{Introduction}

A \df{truncated moment sequence (tms)}
in $n$ variables and of degree $d$ is a finite sequence
$y=(y_\af)$ indexed by nonnegative integer vectors
$\af:=(\af_1,\ldots,\af_n) \in \N^n$ with
$|\af|:=\af_1+\cdots+\af_n\leq d$.
Let $K$ be a semialgebraic set defined as
\be \label{def:K}
K = \left\{ x\in\re^n: \,
\baray{c}
g_1(x)\geq 0,\, \cdots,\, g_m(x) \geq 0
\earay
\right\},
\ee
with $g_1,\ldots,g_m$ being polynomials in $x$.
We say a tms $y$ {\it admits} a $K$-measure  
(a nonnegative Borel measure supported in $K$) 
if there exists a $K$-measure such that
\be  \label{y-meas}
y_\af = \int_K x^\af \mathit{d} \mu, \qquad
\forall \, \af \in \N^n: |\af| \leq d.
\ee
(Here, denote $x^\af:=x_1^{\af_1}\cdots x_n^{\af_n}$
for $x=(x_1,\ldots,x_n)$ and $\af=(\af_1,\ldots,\af_n)$.)
If \reff{y-meas} holds, we say $\mu$ is a representing measure for $y$
and $y$ admits the measure $\mu$.
The \df{truncated $K$-moment problem (TKMP)} is:
How to check if a tms $y$ admit a $K$-measure?
If it admits one, how to get a representing measure?
When $K=\re^n$, TKMP is referred to as
the truncated moment problem (TMP). Let
\def\msM{\mathscr{M}}
\be \label{m-sp:M(n,d)}
\msM_{n,d} :=
\big\{y=(y_{\af}): \af \in \N^n, |\af| \leq d \big\},
\ee
\index{$\mathscr{M}_{n,d}$}
and denote by \ind{$meas(y,K)$}
the set of all $K$-measures admitted by $y$.
Let
\be \label{meas:K-d}
\mathscr{R}_d(K) :=
\big\{y \in \mathscr{M}_{n,d}:  \,meas(y,K) \ne \emptyset \big\}.
\ee
A measure with finite support is called a {\it finitely atomic} measure.
A measure $\mu$ is called \df{$r$-atomic} if $|\supp{\mu}|=r$.

\subsection{Background}

\ind{Bayer and Teichmann} \cite{BT} proved an important result: \
{\it a tms $y\in \mathscr{M}_{n,d}$ admits a $K$-measure $\mu$
if and only if it admits a $K$-measure $\nu$ with
$|\supp{\nu}|\leq \binom{n+d}{d}$.}
Let $\re[x]:=\re[x_1,\ldots,x_n]$ be the ring of polynomials in
$(x_1,\ldots,x_n)$ with real coefficients,
and $\re[x]_d$ be the set of polynomials in $\re[x]$ whose
degrees are at most $d$.
Every tms $y\in \mathscr{M}_{n,d}$ defines
a \df{Riesz functional $\mathscr{L}_y$} acting on $\re[x]_d$ as
\be  \label{def:Riesz}
\mathscr{L}_y \Big( \sum_{|\af|\leq d} p_\af x^\af \Big)
:= \sum_{|\af|\leq d} p_\af y_\af.
\ee
For convenience, sometimes we also denote
\ind{$\langle p, y \rangle := \mathscr{L}_y(p)$}.
Let $P_d(K)$ denote the set of all polynomials in $\re[x]_d$
that are nonnegative on $K$.
A necessary condition for $y$ to belong to $\mathscr{R}_d(K)$
is that $\mathscr{L}_y$ is \df{$K$-positive}, i.e.,
\[
\mathscr{L}_y(p) \geq 0 \quad \forall \,
p \in P_d(K).
\]
This is because
\[
\mathscr{L}_y(p) = \int_K p \mathit{d}\mu \geq 0 \quad
\forall \, p \in P_d(K), \, \forall \, \mu \in meas(y,K).
\]
Indeed, when $K$ is compact, $\mathscr{L}_y$ being $K$-positive
is also sufficient for $y$ to belong to $\mathscr{R}_d(K)$.
This was implied by the proof of Tchakaloff's Theorem \cite{Tch}.
A condition stronger  than $\mathscr{L}_y$ being $K$-positive
is $\mathscr{L}_y$ being \df{strictly $K$-positive}, i.e.,
\[
\mathscr{L}_y(p) > 0 \quad \forall \,
p \in P_d(K), \, p \not\equiv 0 \mbox{ on } K.
\]
As will be shown in Lemma~\ref{lm:Rz:K>0}, when $K$ has nonempty interior,
a tms $y$ belongs to the interior of $\mathscr{R}_d(K)$ if and only if
its Riesz functional $\mathscr{L}_y$ is strictly $K$-positive.
Typically, it is quite difficult to check whether $\mathscr{L}_y$ is $K$-positive
or strictly $K$-positive.

A weaker but more easily checkable condition is that the localizing matrix
of a tms is positive semidefinite if it admits a $K$-measure.
If a symmetric matrix $X$ is positive semidefinite (resp., definite),
we denote $X\succeq 0$ (resp., $X\succ 0$).
For a tms $y$ of degree $2k$, let $M_k(y)$ be
the symmetric matrix linear in $y$ such that
\[
\mathscr{L}_y(p^2) = p^T  M_k(y)   p \qquad
\forall \, p\in \re[x]_k.
\]
(For convenience of notation, we still use $p$ to
denote the vector of coefficients of $p(x)$ in the graded lexicographical ordering.)
The matrix $M_k(y)$ is called a \df{$k$-th order moment matrix}.
For $h \in \re[x]_{2k}$, define the new tms $h\ast y$ such that
\[
\mathscr{L}_y(hq) \, = \,  \mathscr{L}_{h\ast y}(q)
\quad \forall \, q \in \re[x]_{2k-\deg(h)}.
\]
The tms $h\ast y$ is known as a shifting of $y$.
If $\deg(hp^2)\leq 2k$, then it holds that
\index{$L_h^{(k)}$ }
\be \label{df:Loc-Mat}
\mathscr{L}_y(hp^2) \, = \, p^T
\left( M_{k-\lceil \deg(h)/2 \rceil}(h\ast y) \right)   p.
\ee
The matrix $M_{k-\lceil \deg(h)/2 \rceil}(h\ast y)$ is called a
{\it $k$-th order localizing matrix} of $h$.
%
%
%

If a tms $y$ belongs to $\mathscr{R}_{2k}(K)$,
then for $i = 0, 1, \ldots, m$ (denote $g_0:= 1$)
\be  \label{loc-M>=0}
M_{k-d_i}(g_i \ast y) \succeq  0 \quad
\mbox{ where } d_i = \lceil \deg(g_i)/2\rceil.
\ee
This is because for every polynomial $p$ with $\deg(g_ip^2)\leq 2k$ we have
\[
p^T M_{k-d_i}(g_i \ast y)p  = \mathscr{L}_y(g_ip^2)
= \int_K g_ip^2 \mathit{d}\mu \geq 0
\]
for all $\mu \in meas(y,K)$.
Thus, \reff{loc-M>=0} is a necessary condition for $y$
to belong to $\mathscr{R}_{2k}(K)$.
In general, \reff{loc-M>=0} is not sufficient for $y$
to belong to $\mathscr{R}_{2k}(K)$.
However, in addition to \reff{loc-M>=0}, if $y$ is also \df{flat}, i.e.,
$y$ satisfies the rank condition
\be \label{cond:FEC}
\rank \, M_{k-d_g}(y) = \rank \, M_k(y),  \qquad
\ee
then $y$ belongs to $\mathscr{R}_{2k}(K)$.
The above integer $d_g$ is defined as
\be \label{df:dg}
d_g  := \max_{1\leq i \leq m}\{1, \lceil \deg(g_i)/2 \rceil \} .
\ee
The next important result is due to
Curto and Fialkow. \index{Curto and Fialkow Theorem}

\begin{theorem}[\cite{CF05}] \label{thm-CF:fec}
Let $K$ be defined in \reff{def:K} and $d$ be even.
If a tms $y\in \mathscr{M}_{n,d}$ satisfies \reff{loc-M>=0} and
the flat extension condition \reff{cond:FEC} for $k=d/2$,
then $y$ admits a unique $\rank\,M_k(y)$-atomic $K$-measure.
\end{theorem}

A nice exposition and proof of Theorem~\ref{thm-CF:fec}
can be found in Laurent \cite{Lau05}.
Flat extensions are not only used in solving TKMPs,
but also frequently used for solving polynomial optimization problems
(cf. \cite{HenLas05,LasBok,Lau05,Lau}).

There are other necessary or sufficient conditions
for a tms $y$ to belong to $\mathscr{R}_d(K)$, like recursively generated relations.
Most of them are about the case of $M_k(y)$ being singular.
We refer to \cite{CF91,CF96,CF98,CF02,CF05,CF052,CF11,F082}
for the work in this area.
When $M_k(y)$ is positive definite, there is little work on TKMPs,
except for the cases $n=1$ (cf. \cite{CF91}), or $d=2$ or $(n,d)=(2,4)$ (cf. \cite{FiNi1}).

In solving TKMPs, flat extensions play an important role.
It is interesting to know when and how a tms $y \in \mathscr{M}_{n,d}$
is extendable to a flat tms $w \in \mathscr{M}_{n,2k}$ with $2k\geq d$.
We say that $y$ extends to $w$ (or $w$ is an extension of $y$)
if $y_\af = w_\af$ for all $|\af|\leq d$.
There is an important result due to
Curto and Fialkow \cite{CF05} about this.

\begin{theorem}[\cite{CF05}]\label{CF:flat-extn}
Let $K$ be defined in \reff{def:K}.
Then a tms $y \in \mathscr{M}_{n,d}$ admits a $K$-measure if and only if
it is extendable to a flat tms $w \in \mathscr{M}_{n,2k}$ with $2k\geq d$
and $M_{k-d_i}(g_i \ast w)\succeq 0$ for $i=0,1,\ldots,m$.
\end{theorem}

%
%

In view of the above result, the following questions arise naturally:
i) If a tms $y$ admits no $K$-measure,
how do we get a certificate for the nonexistence? \
ii) If $y$ admits a $K$-measure, how do we get
a representing measure? \
iii) Preferably, if a tms admits a $K$-measure,
how can we get a finitely atomic representing measure?

\subsection{Contributions}
\label{sec:contrib}

This paper focuses on the questions above.
We propose a general semidefinite programming (SDP) approach for solving TKMPs.
When $K$ is compact as in \reff{def:K}, we have the following results:
\bnum
\item
\label{it:goal1}
 Whether a tms admits a $K$-measure or not
can be checked via solving a sequence of semidefinite programs $\{ \tt{(SDP)k} \}$
(see \reff{max-lmd:extn-y} or \reff{max-lmd:K-qmod}).
\begin{enumerate}

\item A tms $y$ admits a $K$-measure if and only if
the optimal value of  {\tt (SDP)k} is nonnegative for all $k$.
Consequently, when $y$ admits no $K$-measure,
the optimal value of {\tt (SDP)k} will be negative for some $k$,
which gives a certificate for the nonexistence of a representing measure.
See Theorems~\ref{sdp-pr:K-mo} and \ref{put-infeas}.

\item When $y$ admits a $K$-measure,
we show how to construct such a measure for $y$
via solving {\tt (SDP)k} for a certain $k$ and using the flat extension condition.
This works under some necessary and sufficient conditions
(they are not far away from each other).
See Theorems~\ref{sdp-pr:K-mo} and \ref{thm:suf-Kmom}.

\end{enumerate}

\item
We propose a practical SDP method for finding flat extensions,
and thus provide a way for constructing finitely atomic
representing measures.
It is based on optimizing linear functionals (see \reff{min-tr:qmod})
on moment matrices
and consists of a sequence of SDP problems.
In our computational experiences, the method always produced a flat extension
if it exists. We derive a bit of supporting theory for this fact.
When a tms admits a $K$-measure,
we prove that for a dense subset of linear functionals,
the method asymptotically produces a flat extension.
When a tms does not admit a $K$-measure,
we prove that this sequence of SDP problems
will become infeasible after some steps.
This method is also applicable when $K$ is noncompact.
See Subsection~4.1 and Theorem~\ref{thm:cvg-flat}.

\enum
%
The results described in (1) are given in Section~\ref{sec:TKMPcompact},
and those in (2) are in Section~\ref{sec:heuristic}.
Section~2 presents some background
for proving these results.

\section{Preliminaries}
\setcounter{equation}{0}

\subsection{Notation}
The symbol $\N$ (resp., $\re$) denotes the set of nonnegative integers
(resp., real numbers),
and $\re_+^n$ denotes the nonnegative orthant of $\re^n$.
For $t\in \re$, $\lceil t\rceil$ (resp., $\lfloor t\rfloor$)
denotes the smallest integer not smaller
(resp., the largest integer not greater) than $t$.
The \df{$[x]_d$} denotes the column vector of all monomials
with degrees not greater than $d$:
\[
[x]_d = [\, 1 \quad  x_1 \quad \cdots \quad x_n \quad x_1^2 \quad
x_1x_2 \quad \cdots \cdots
\quad x_1^d \quad x_1^{d-1}x_2 \quad \cdots \cdots \quad x_n^d \,]^T.
\]
For a set \df{$S \subseteq \re^n$}, $|S|$ denotes its cardinality.
The symbol $\mbox{int}(\cdot)$ denotes the interior of a set.
For a matrix $A$, $A^T$ denotes its transpose;
if $A$ is symmetric, $\lmd_{min}(X)$
denotes its minimum eigenvalue.
For $u\in \re^N$, $\| u \|_2 := \sqrt{u^Tu}$ denotes the standard Euclidean norm,
and $B(u,r) :=\{x\in \re^n:\, \|x-u\|_2 \leq r\}$ denotes the closed ball
with center $u$ and radius $r$.
The $\bullet$ denotes the standard Frobenius inner product in matrix spaces.
For a matrix $A$, $\|A\|_F$ denotes the Frobinus norm of $A$, i.e.,
$\|A\|_F = \sqrt{Trace(A^TA)}$.
%
%
The zero set of a polynomial $q$ is denoted by $\mc{Z}(q)$.
For a tms $w$, $w|_t$ denotes the subsequence of $w$
whose indices have degrees not greater than $t$, i.e.,
$w|_t$ is a truncation of $w$ with degree $t$.
For a measure $\mu$, $\supp{\mu}$ denotes its support.
A polynomial $f\in \re[x]$ is said to be \df{sum of squares (SOS)}
if there exist $f_1,\ldots,f_k \in \re[x]$ such that
$f=f_1^2+\cdots+f_k^2$.
The set of all SOS polynomials in $n$ variables and of degree $d$
is denoted by $\Sig_{n,d}$.
%
%
The symbol $\mathscr{C}(K)$ denotes the space of all functions that are
continuous on a compact set $K$, and $\|\cdot\|_\infty$
denotes its standard $\infty$-norm.

\subsection{Truncated moments}

For a compact set $K$, a tms
admits a measure supported in $K$ if and only if
its Riesz functional is $K$-positive.
This result is implied in the proof of Tchakaloff's Theorem \cite{Tch}
(cf. \cite{FiNi1}). \index{Tchakaloff's Theorem}

\begin{theorem}[Tchakaloff] \label{thm:Tchak}
Let $K$ be a compact set in $\re^n$. A tms $y$
admits a $K$-measure if and only if its
Riesz functional $\mathscr{L}_y$ is $K$-positive.
\end{theorem}

In the following, we characterize the interior of $\mathscr{R}_d(K)$
via strict $K$-positivity of Riesz functionals.

\begin{lem} \label{lm:Rz:K>0}
Let $K \subseteq \re^n$ be a set with $\mbox{int}(K)\ne\emptyset$ and
$y$ be a tms in $\mathscr{M}_{n,d}$.
Then $y$ belongs to $\mbox{int}\left(\mathscr{R}_d(K)\right)$
if and only if $\mathscr{L}_y$ is strictly $K$-positive.
\end{lem}
\begin{proof}
$``\Rightarrow"$ Suppose $y$ belongs to $\mbox{int}\left(\mathscr{R}_d(K)\right)$.
Let $\zeta \in \mathscr{M}_{n,d}$ be the tms generated 
by the standard Gaussian measure restricted to $K$.
Then, for every $p\in P_d(K)$ with $p|_K \not\equiv 0$,
we must have $\mathscr{L}_{\zeta}(p) > 0$ because $\mbox{int}(K) \ne \emptyset$.
If $\eps>0$ is small enough, the tms $v:=y-\eps \zeta$ belongs to $\mathscr{R}_d(K)$
and it holds that
\[
\mathscr{L}_y(p) = \mathscr{L}_v(p) + \eps \mathscr{L}_{\zeta}(p)
\geq \eps \mathscr{L}_{\zeta}(p) > 0.
\]
This means that $\mathscr{L}_y$ is strictly $K$-positive.

$``\Leftarrow"$ Suppose $\mathscr{L}_y$ is strictly $K$-positive.
Lemma~2.3 of \cite{FiNi1} implies that, for some $\dt>0$, the
$\mathscr{L}_{w}$ is strictly $K$-positive for all $w\in \mathscr{M}_{n,d}$
with $\|w-y\|_2 < \dt$. Then, by Theorem~2.4 of \cite{FiNi1},
every such a tms $w$, including $y$, belongs to $\mathscr{R}_d(K)$.
This implies that $y$ belongs to $\mbox{int}\left(\mathscr{R}_d(K)\right)$.
\end{proof}

A tms $w \in \mathscr{M}_{n,2k}$ generates the moment matrix $M_k(w)$.
Recall that, for a polynomial $f$, we still denote by $f$
the vector of its coefficients.
The vector $f$ is indexed by exponents of $x^\af$.
The matrix vector product $M_k(w) f$
is defined in the usual way, i.e.,
\[
M_k(w) f \, = \, \mathscr{L}_w\big(f[x]_k\big).
\]
We say $p\in \ker M_k(w)$ if $M_k(w)p=0$, i.e., $\mathscr{L}_w(p x^\af)=0$
for every $|\af|\leq k$.

An \df{ideal of $\re[x]$} is a subset $ I \subseteq \re[x]$
such that $I + I \subseteq I$ and $p\cdot q \in I$ for every $p\in I$ and $q \in \re[x]$.
Given $p_1,\ldots,p_m \in \re[x]$,
denote by $\langle p_1,\cdots,p_m \rangle $ the ideal
generated by $p_1,\ldots, p_m$.
If $I=\langle p_1,\ldots, p_m \rangle$ and every $p_i \in \ker M_k(w)$,
then $I \subseteq  \langle \ker M_k(w) \rangle$,
the ideal generated by polynomials in $\ker M_k(w)$.

\begin{lem} \label{lm:pq-ker}
Let $w \in \mathscr{M}_{n,2k}, h, p \in \re[x]$ be such that
$
H:=M_{k- \lceil \deg(h)/2 \rceil}(h\ast w)\succeq 0.
$
If $p,q$ are polynomials with $p \in \ker H$ and
\[ \deg(pq) \leq k-\lceil \deg(h)/2 \rceil-1, \]
then we have $pq \in \ker H$.
\end{lem}
\begin{proof}
Let $z=h\ast w$. Then $z$ is a tms in $\mathscr{M}_{n,2k-\deg(h)}$
and the moment matrix $M_{k-\lceil \deg(h)/2 \rceil}(z)\succeq 0$.
The conclusion is implied by Lemma~3.5 of \cite{LLR08}
(also see Lemma~5.7 of \cite{Lau} or Theorem 7.5 of \cite{CF96}).
\end{proof}

%

\subsection{Quadratic module, preordering and semidefinite programming}

For the semialgebraic set $K$ defined in \reff{def:K},
its $k$-th \df{truncated quadratic module $Q_k(K)$} and
\df{truncated preordering $Pr_k(K)$} are respectively defined as
\be \label{qmod:K}
Q_k(K) :=
\left\{
\left. \sum_{i=0}^m g_i \sig_i \right|
\baray{c}
\mbox{ each } \deg(\sig_ig_i) \leq 2k \\
\mbox{ and } \sig_i \mbox{ is SOS}
\earay
\right\},
\ee
\be \label{pre:K}
Pr_k(K) :=
\left\{
\left.\sum_{ \nu \in \{0,1\}^m }^m \sig_\nu g_\nu \right|
\baray{c}
\deg(\sig_\nu g_\nu) \leq 2k \\ \mbox{ each } \sig_\nu \mbox{ is SOS}
\earay
\right\}.
\ee
($g_\nu := g_1^{\nu_1}\cdots g_m^{\nu_m}$.)
The quadratic module and preordering of $K$ are then defined respectively as
\[
Q(K) = \bigcup_{k\geq 0} Q_k(K), \qquad
Pr(K) = \bigcup_{k\geq 0} Pr_k(K).
\]
The definitions of $Q_k(K)$ and $Pr_k(K)$ depend on the set of
defining polynomials $g_1,\ldots, g_m$, which are not unique for $K$.
Throughout the paper, when $Q_k(K)$ or $Pr_k(K)$ is used,
we assume that $g_1,\ldots, g_m$ are clear in the context.
In \reff{def:K}, if $K$ is defined by using polynomial equalities, like $h(x)=0$,
then it can be replaced by two inequalities $h(x)\geq 0$ and $-h(x)\geq 0$.

\begin{theorem}\label{thm:PutSch}
Let $K$ be as in \reff{def:K} and $p\in \re[x]$ be strictly positive on $K$.

\bit
\item [(i)] (Schm\"{u}dgen, \cite{Smg}) If $K$ is compact, then we have $p \in Pr(K)$.

\item [(ii)] (Putinar, \cite{Put}) If the archimedean condition holds for $K$
(a set $\{x:q(x)\geq 0\}$ is compact for some $q \in Q(K)$),
then we have $p \in Q(K)$.
\eit
\end{theorem}
\index{Putinar's Theorem}
\index{Schmudgen's Theorem}

\begin{theorem}[Real Nullstellensatz]
\label{RealNul}
Let $K$ be defined in \reff{def:K}.
If $f \in \re[x]$ vanishes identically on $K$, then
$-f^{2\ell} \in Pr(K)$ for some integer $\ell \geq 1$.
\end{theorem}
\noindent {\it Remark:}
Theorem~\ref{RealNul} is a special case of the so-called
Positivstellensatz \cite{Stg74}, and the set $K$ there
does not need to be compact.

Let $I(K)$ be the vanishing ideal of $K$, i.e.,
\[
I(K) :=    \{ h\in \re[x]\,:\, h(u)=0 \,\,\,\forall \,\, u \in K\}.
\]
By Theorem~\ref{RealNul}, if $f \in I(K)$, then
$-f^{2\ell} \in Pr_k(K)$ for some $k,\ell$.
This fact will be used in our proofs later.

The sets $Q_k(K)$ and $Pr_k(K)$ are convex cones.
Their dual cones lie in the space $\mathscr{M}_{n,2k}$.
The dual cone of $Pr_k(K)$ is defined as
\[
Pr_k(K)^* = \{ y \in \mathscr{M}_{n,2k}: \,
\langle p, y\rangle \geq 0 \,\,\, \forall p \in Pr_k(K)\}.
\]
The dual cone $Q_k(K)^*$ is defined similarly.
It is known that (cf. \cite{Las01, LasBok})
\begin{align*}
Q_k(K)^* &= \left\{ y \in \mathscr{M}_{n,2k}: \, M_{k-d_i}(g_i\ast y) \succeq 0\,\,\,
\mbox{for } \, i=0,1,\ldots, m \right\}, \\
Pr_k(K)^* &= \left\{ y \in \mathscr{M}_{n,2k}: \,
M_{k-d_\nu}(g_\nu \ast y) \succeq 0 \,\, \,
\forall \, \nu \in \{0,1\}^m \right\}.
\end{align*}
($d_\nu := \lceil \deg(g_\nu)/2 \rceil$.)
Given tms' $c,a_1,\ldots, a_t \in \mathscr{M}_{n,d}$
and scalars $b_1,\ldots,b_t$, consider the
linear conic optimization problem
\be \label{sdp:pro}
\left\{
\baray{rl}
\underset{p}{\min} & \langle p, c \rangle \\
s.t. &  \langle p, a_i \rangle  = b_i \, (1\leq i \leq t),
\quad \  p \in \re[x]_d \cap Pr_k(K).
\earay \right.
\ee
Its dual optimization problem is
\be  \label{sdp:mom}
\left\{
\baray{rl}
\underset{y \in \mathscr{M}_{n,2k}}{\max} & b_1\lmd_1 + \cdots + b_t \lmd_t \\
s.t. & w|_d = c - \lmd_1 a_1 - \cdots - \lmd_t a_t, \\
& M_{k-d_\nu}(g_\nu \ast w) \succeq 0\, \forall \nu \in \{0,1\}^m.
\earay\right.
\ee
The optimization problems \reff{sdp:pro} and \reff{sdp:mom} are
reducible to SDP problems (cf. \cite{Las01,LasBok,Lau}).
Any objective value of a feasible solution of  \reff{sdp:pro} (resp., \reff{sdp:mom})
is an upper bound (resp., lower bound) for the optimal value of the other one
(this is called \df{weak duality}).
If one of them has an interior point
(for \reff{sdp:pro} it means that there is a feasible $p$
lying in the interior of $\re[x]_d \cap Pr_k(K)$,
and for \reff{sdp:mom} it means that there is a feasible
$w$ satisfying every $M_{k-d_\nu}(g_\nu \ast w) \succ 0$),
then the other one has an optimizer and they have the same optimal value
(this is called \df{strong duality}).
Similar is true if $Pr_k(K)$ is replaced by $Q_k(K)$.
We refer to \cite{Las01,LasBok,Lau} for properties of SDPs arising
from moment problems and polynomial optimization.
In \cite{Las08} Lasserre proposed a  semidefinite programming approach
for solving the generalized problem of moments.
We refer to \cite{WSV} for more about SDP.

\section{Checking Existence of Representing Measures}
\label{sec:TKMPcompact}
\setcounter{equation}{0}

This section discusses TKMPs when $K$ is
a compact semialgebraic set defined in \reff{def:K}.
It gives a semidefinite approach for
checking if a tms admits a representing measure
and proves its properties alluded to in \S \ref{sec:contrib}(1).

\subsection{A certificate via semidefinite programming}

Our semidefinite approach for solving TKMPs exploits
the following basic fact.

\begin{pro}  \label{certf:meas-K}
Let $y$ be a tms in $\mathscr{M}_{n,d}$ and
$K$ be a compact set defined by \reff{def:K}.
Then $y$ admits no $K$-measure if and only if there exists
a polynomial $p$ such that
\[ \langle p, y \rangle < 0, \quad  p \in \re[x]_d \cap Pr(K). \]
Under the archimedean condition for $K$, the above is also true
if $Pr(K)$ is replaced by $Q(K)$.
\end{pro}
\begin{proof}
(i) The ``if'' direction is obvious. It suffices to prove the ``only if'' direction.
Suppose $y$ does not belong to $\mathscr{R}_d(K)$. Then, by Theorem~\ref{thm:Tchak},
the Riesz functional $\mathscr{L}_y$ must achieve a negative value on $P_d(K)$,
say $\hat{p} \in P_d(K)$, such that $\mathscr{L}_y(\hat{p}) <0$.
Then, for a small enough $\eps>0$, the polynomial $p:=\hat{p}+\eps$
also satisfies $\mathscr{L}_y(p) <0$.
Since $p$ is strictly positive on the compact set $K$,
by Theorem~\ref{thm:PutSch}(i), we have $p\in Pr(K)$.
This $p$ satisfies $\langle p, y\rangle <0$.
The proof is same when the archimedean condition holds
(applying Theorem~\ref{thm:PutSch}(ii)).
\end{proof}

In the following, we show how to apply Proposition~\ref{certf:meas-K}
to check whether a tms $y$ belongs to $\mathscr{R}_d(K)$ or not.
Choose a tms $\xi \in \mathscr{R}_d(K)$ such that
$\mathscr{L}_{\xi}$ is strictly $K$-positive.
For an integer $k\geq d/2$, consider the optimization problem
\be  \label{min<p,y>:Pr}
\tilde{\lmd}_k:= \, \underset{p}{\min} \quad \langle p, y \rangle \quad
s.t. \quad  \langle p, \xi \rangle  = 1, \, p \in \re[x]_d \cap Pr_k(K).
\ee
Its dual optimization problem is ($d_\nu := \lceil \deg(g_\nu)/2 \rceil$)
\be  \label{max-lmd:extn-y}
\left\{
\baray{rl}
\lmd_k := \, \underset{\lmd \in \re, w \in
 \mathscr{M}_{n,2k} }{\max} & \lmd \\
s.t. \qquad &  w|_d =y - \lmd \xi, \\
& M_{k-d_\nu}(g_\nu \ast w) \succeq 0 \,\, \forall \nu \in \{0,1\}^m.
\earay \right.
\ee
Both the primal \reff{min<p,y>:Pr} and dual \reff{max-lmd:extn-y}
are reducible to SDP problems,
and they are parameterized by an order $k$.
They can be solved efficiently,
e.g., by software {\tt GloptiPoly} \cite{GloPol3} and {\tt SeDuMi} \cite{sedumi}.

A nonnegative Borel measure $\mu$ is called \df{$(K,d)$-semialgebraic}
if $\supp{\mu} \subseteq K$ and there exists
$q \in \re[x]_d \cap Pr(K)$ such that
\[
\supp{\mu} \, \subseteq \, \big\{ u \in K:\, q(u) =0 \big\},
\quad q \not\equiv 0 \mbox{ on } K.
\]
If a measure $\mu$ is not $(K,d)$-semialgebraic, then
\[
\supp{\mu} \subseteq K \cap \mc{Z}(q),\,
q \in \re[x]_d \cap Pr(K)
\quad \Longrightarrow \quad q \equiv 0 \mbox{ on } K.
\]
Note that not every $K$-measure is $(K,d)$-semialgebraic.
For instance, for $K$ being the unit ball,
the probability measure uniformly distributed on $K$
is not $(K,d)$-semialgebraic,
because every polynomial vanishing on it must be identically zero.
So, being $(K,d)$-semialgebraic is a bit restrictive for a measure $\mu$,
as it implies $\supp{\mu}$ has Lebesgue measure zero when $K$ has nonempty interior.

\begin{theorem} \label{sdp-pr:K-mo}
Assume the set $K$ defined in \reff{def:K} is compact,
the tms $\xi$ belongs to $\mathscr{R}_d(K)$ and
its Riesz functional $\mathscr{L}_\xi$ is strictly K-positive.
Let $y, \lmd_k, \tilde{\lmd}_k$ be as above.
Then the sequence $\{\lmd_k\}$ is monotonically decreasing.
Let $\lmd_\infty := \lim_{k\to \infty} \lmd_k \in \re\cup \{-\infty\}.$
We also have:

\bit

\item [(i)] For $k$ big enough, $\tilde{\lmd}_k = \lmd_k$
and \reff{max-lmd:extn-y} has a maximizer.

\item [(ii)] The tms $y$ belongs to $\mathscr{R}_d(K)$ if and only if
$\lmd_k \geq 0$ for all $k$.
If $\lmd_\infty > -\infty$, then the shifted tms $\hat{y}:=y - \lmd_\infty \xi$
belongs to $\mathscr{R}_d(K)$.

\item [(iii)] Assume $\mbox{int}(K)\ne \emptyset$.
Then, $\lmd_\infty >-\infty$,
and $\lmd_k = \lmd_\infty$ for some $k$ if and only if
there exists a $(K,d)$-semialgebraic measure $\mu \in meas(\hat{y},K)$.

\item [(iv)] Suppose, for some $k_0$, \reff{max-lmd:extn-y}
has an optimizer $(\lmd_{k_0}, w^*)$ with $\lmd_{k_0} \geq 0$
and $w^*|_t(t\geq d)$ flat.
Then $y$ belongs to $\mathscr{R}_d(K)$ and $\lmd_k = \lmd_{k_0}$ for all $k\geq k_0$.

\eit
\end{theorem}

\begin{proof}
Since the feasible set of \reff{max-lmd:extn-y}
is shrinking as $k$ increases, the sequence of optimal values $\{\lmd_k\}$
must be monotonically decreasing.
So its limit $\lmd_\infty := \lim_{k\to \infty} \lmd_k$
exists (would possibly be $-\infty$).

(i) Since $K$ is compact, there exists $R>0$ big enough such that,
for all $\af \in \N^n$ with $|\af| \leq d$,
the polynomial $R-1+x^\af$ is positive on $K$ and belongs to
$Pr_k(K)$ for some $k$, by Theorem~\ref{thm:PutSch}(i).
For $\eps>0$ small enough, we also have $R+c_0+(1+c_1)x^\af \in Pr_k(K)$
whenever $|c_0|,|c_1|<\eps$, for all $\af \in \N^n$ with $|\af| \leq d$.
This implies that the polynomial
$\hat{p}:= \sum_{|\af|\leq d} (R+x^\af)$ lies in
the interior of the cone $\re[x]_d\cap Pr_k(K)$.
Thus, for $k$ big enough, the primal problem \reff{min<p,y>:Pr} has a feasible point
lying in the interior of $\re[x]_d \cap Pr_k(K)$.
Hence, the primal and dual optimization problems \reff{min<p,y>:Pr} and \reff{max-lmd:extn-y}
have the same optimal value, and \reff{max-lmd:extn-y} has a maximizer
(see the comments following \reff{sdp:mom}).

(ii) If the tms $y$ belongs to $\mathscr{R}_d(K)$, then $\lmd_k\geq 0$ for all $k$,
because $\lmd=0$ is feasible in \reff{max-lmd:extn-y} for every $k$.
If $y$ does not belong to $\mathscr{R}_d(K)$, by Proposition~\ref{certf:meas-K},
there exists $p\in \re[x]_d \cap Pr_k(K)$
such that $\langle p, y \rangle <0$, for some $k$.
Generally, we can assume $p \not\equiv 0$ on $K$
(otherwise, replace $p$ by $p+\eps$ for a tiny $\eps>0$).
By the strict $K$-positivity of $\xi$,
we can normalize $p$ as $\langle p, \xi \rangle =1$.
Hence, $\lmd_k \leq \tilde{\lmd}_k < 0$.
This shows that the first statement of item (ii) is true.

The second statement of item (ii) is implied by the first one.
This is because, in \reff{max-lmd:extn-y}, if $y$ is replaced by $\hat{y}$,
then the resulting optimal values are all nonnegative.

(iii) By item (i), for big $k$, $\lmd_k=\tilde{\lmd}_k$
is at least the optimal value of
\be  \label{min<p,y>:Pd}
\min_p \quad \langle p, y \rangle
\quad s.t. \quad \langle p, \xi \rangle  = 1,  p \in P_d(K).
\ee
Since $\mbox{int}(K)\ne\emptyset$, $p\not\equiv 0$ on $K$
if and only if $p$ is not identically zero.
Thus, the strict $K$-positivity of $\mathscr{L}_{\xi}$ implies
\be  \label{<p,xi>=ep1}
1 = \langle p, \xi \rangle \geq \eps_1 \|p\|_2, \mbox{  for some  } \eps_1>0.
\ee
So, the feasible set of \reff{min<p,y>:Pd} is compact
and has a finite minimum value $\hat{\lmd}>-\infty$.
By the monotonicity of $\{ \lmd_k \}$, we know
$\lmd_k \geq \hat{\lmd}$ for every $k$
and its limit $\lmd_\infty$ is finite.

``if" direction:
If there is a $(K,d)$-semialgebraic measure $\mu \in meas(\hat{y},K)$,
then for some $k$ we can find $ 0 \ne q \in \re[x]_d \cap Pr_k(K)$ such that
$\supp{\mu}\subseteq K\cap\mc{Z}(q)$. Normalize $q$ as
$\langle q, \xi \rangle = 1$, then
\[
\langle q, y \rangle = \langle q, \hat{y} \rangle + \lmd_\infty
= \int_K q d\mu + \lmd_\infty  = \lmd_\infty.
\]
So, the polynomial $q$ is feasible for \reff{min<p,y>:Pr},
and $\tilde{\lmd}_k \leq \lmd_\infty$.
By weak duality, $\lmd_k \leq \tilde{\lmd}_k \leq \lmd_\infty$.
Since $\lmd_k \geq \lmd_\infty$ for all $k$,
we must have $\lmd_k = \lmd_\infty$.

``only if" direction:  If $\lmd_{k_0} = \lmd_\infty$ for some $k_0$,
then $\lmd_k = \lmd_\infty$ for all $k\geq k_0$. Fix such a $k$.
Since $\mbox{int}(K)\ne\emptyset$, the feasible set of \reff{min<p,y>:Pr}
is compact (by \reff{<p,xi>=ep1}) and has a minimizer $p^* \ne 0$.
As shown in (i), if $k$ is big enough, $\tilde{\lmd}_k = \lmd_k$ and
\reff{max-lmd:extn-y} has a maximizer,
say, $(\lmd_k, w^*)$. Then, for every $\mu \in meas(\hat{y},K)$, it holds that
\[
0 = \langle p^*, y - \lambda_k \xi \rangle=
\langle p^*, \hat{y} \rangle = \int_K p^* \mathit{d}\mu.
\]
Since $p^*$ is nonnegative on $K$, $\supp{\mu} \subseteq \mc{Z}(p^*)$
and $\mu$ is $(K,d)$-semialgebraic.


(iv) Note that $y = w^*|_d + \lmd_{k_0} \xi$ and $\xi$ admits a $K$-measure,
say, $\nu$. If $w^*|_t(t\geq d)$ is flat, then $w^*|_t$
admits a finitely atomic $K$-measure $\mu$,
by Theorem~\ref{thm-CF:fec}. Since $\lmd_{k_0} \geq 0$,
we know $\mu+\lmd_{k_0} \nu$ is a representing $K$-measure for $y$.
So, $y$ belongs to $\mathscr{R}_d(K)$.

Clearly, for every $k$, the projection of the feasible set of \reff{max-lmd:extn-y}
into the $\lmd$-space contains the set
\[
F = \big\{ \lmd: \, y - \lmd \xi \in \mathscr{R}_d(K)  \big\}.
\]
Thus, every optimal value $\lmd_k$ is greater than or equal to the optimal value
\[
\lmd^*  :=  \max \quad \lmd \quad \mbox{ s.t. } \quad \lmd \in F.
\]
So, we know $\lmd_{k_0} \in F$ from the above.
Hence, $\lmd_k \geq \lmd^* \geq \lmd_{k_0}$ for all $k$.
By the decreasing monotonicity of $\{\lmd_k\}$, we know $\lmd_k = \lmd_{k_0}$
for all $k \geq k_0$.
\end{proof}

When $\mbox{int}(K)\ne\emptyset$, as we can see in the proof of item (iii)
of Theorem~\ref{sdp-pr:K-mo}, a measure in $meas(\hat{y},K)$ is $(K,d)$-semialgebraic
if and only if all measures in $meas(\hat{y},K)$ are $(K,d)$-semialgebraic.

In \reff{max-lmd:extn-y},
all the cross products $g_\nu := g_1^{\nu_1}\cdots g_m^{\nu_m}$
are used, which might be inconvenient in applications.
So, we consider a simplified version of \reff{max-lmd:extn-y}:
\be  \label{max-lmd:K-qmod}
\left\{
\baray{rl}
\widehat{\lmd}_k := \, \underset{\lmd\in\re, w \in \mathscr{M}_{n,2k}}{\max} & \lmd \\
s.t. \qquad & w|_d =y - \lmd \xi, \\
& M_{k-d_i}(g_i \ast w) \succeq 0 \,(0\leq i\leq m).
\earay \right.
\ee
The dual optimization problem of \reff{max-lmd:K-qmod} is
\be  \label{min<p,y>:qmod}
\widehat{\lmd}_k^{\prm} := \quad \underset{p}{\min} \quad \langle p, y \rangle \quad
s.t. \quad  \langle p, \xi \rangle  = 1, \, p \in \re[x]_d \cap Q_k(K).
\ee
We have the following analogue to Theorem~\ref{sdp-pr:K-mo}.

\begin{theorem} \label{put-infeas}
Let $K,\xi, y, \hat{y}, \lmd_k, \lmd_\infty$ be same as in Theorem~\ref{sdp-pr:K-mo},
and $\widehat{\lmd}_k, \widehat{\lmd}_k^{\prm}$ be as above.
Suppose the archimedean condition holds for $K$.
Then, the sequence $\{ \widehat{\lmd}_k \}$ is monotonically decreasing.
Let $\widehat{\lmd}_\infty := \lim_{k\to \infty} \widehat{\lmd}_k \in \re \cup \{-\infty\}$.
We also have:
\bit

\item [(i)] For $k$ big enough, $\widehat{\lmd}_k = \widehat{\lmd}_k^{\prm}$
and \reff{max-lmd:K-qmod} has a maximizer.

\item [(ii)] The tms $y$ belongs to $\mathscr{R}_d(K)$ if and only if
$\widehat{\lmd}_k\geq 0$ for all $k$.

\item [(iii)]
It always holds that $\widehat{\lmd}_\infty =\lmd_\infty$.
If $\mbox{int}(K)\ne \emptyset$, then $\widehat{\lmd}_\infty > -\infty$, and
$\widehat{\lmd}_k = \widehat{\lmd}_\infty$ for some $k$ if and only if
there exists $\mu \in meas(\hat{y},K)$ such that
\[
\supp{\mu} \subseteq K \cap Z(q) \quad
\mbox{ for some } q\in \re[x]_d \cap Q(K), \, q|_K \not\equiv 0.
\]

\item [(iv)] Suppose, for some $k_0$, the optimization problem \reff{max-lmd:K-qmod}
has an optimizer $(\hat{\lmd}_{k_0}, w^*)$ with $\lmd_{k_0} \geq 0$
and $w^*|_t(t\geq d)$ flat. Then, $y$ belongs to $\mathscr{R}_d(K)$
and $\hat{\lmd}_k  = \hat{\lmd}_{k_0}$ for all $k\geq k_0$.

\eit
\end{theorem}
\noindent{\it Remark:}
By Theorem~\ref{put-infeas} (ii), if $y$ does not belong to $\mathscr{R}_d(K)$,
then $\widehat{\lmd}_k < 0$ for some $k$,
and for $\lmd = 0$ there is no $w \in \mathscr{M}_{n,2k}$ satisfying
\be \label{extn-infea}
w|_d =y, \, M_k(w) \succeq 0, \,
M_{k-d_1}(g_1 \ast w) \succeq 0, \, \ldots, \, M_{k-d_m}(g_m \ast w)  \succeq 0.
\ee
So, we have $y \not\in \mathscr{R}_d(K)$ if and only if
\reff{extn-infea} is infeasible for some $k$.

\begin{proof}[Proof of Theorem~\ref{put-infeas}]
The decreasing monotonicity of $\{ \widehat{\lmd}_k \}$
holds because the feasible set of \reff{max-lmd:K-qmod}
shrinks as $k$ increases.

(i)
As in Theorem~\ref{sdp-pr:K-mo}, under the archimedean condition, for $k$ big enough,
we can similarly show that \reff{min<p,y>:qmod} has a feasible point
lying in the interior of the cone $\re[x]_d \cap Q_k(K)$.
Thus, \reff{max-lmd:K-qmod} and \reff{min<p,y>:qmod}
have the same optimal value, i.e., $\widehat{\lmd}_k = \widehat{\lmd}_k^{\prm}$,
and \reff{max-lmd:K-qmod} has a maximizer when $k$ is sufficiently large.

(ii) This can be proved in a way similar to item (ii) of Theorem~\ref{sdp-pr:K-mo},
under the archimedean condition on $K$.

(iii)
We first show that $\widehat{\lmd}_\infty =\lmd_\infty$.
Clearly, $\widehat{\lmd}_k \geq \lmd_k$ for all $k$.
Thus $\widehat{\lmd}_\infty \geq \lmd_\infty$.
If $\widehat{\lmd}_\infty  = -\infty$, then we are done.
If $\widehat{\lmd}_\infty  > -\infty$, we need to show
$\widehat{\lmd}_\infty = \lmd_\infty$.
If not, seeking a contradiction, we suppose otherwise
$\widehat{\lmd}_\infty > \lmd_\infty$.
Let $\tilde{y}:=y - \widehat{\lmd}_\infty \xi$.
If we replace $y$ by $\tilde{y}$ in \reff{max-lmd:K-qmod},
then all its optimal values are nonnegative,
and thus we have $\tilde{y} \in \mathscr{R}_d(K)$ by item (ii).
However, by Theorem~\ref{sdp-pr:K-mo} (ii), we get
$\tilde{y} \not\in \mathscr{R}_d(K)$, a contradiction.
This is because if we replace $y$ by $\tilde{y}$ in \reff{max-lmd:extn-y},
then its optimal value is $\lmd_k - \widehat{\lmd}_\infty$
which is negative for big $k$.
So, we must have $\widehat{\lmd}_\infty = \lmd_\infty$.

The second statement of item (iii) can be shown in a similar way
as for item (iii) of Theorem~\ref{sdp-pr:K-mo}.
The only difference is to apply Proposition~\ref{certf:meas-K}
with $Q(K)$, under the archimedean condition on $K$.

(iv) The proof is the same as for item (iv) of Theorem~\ref{sdp-pr:K-mo}.
\end{proof}

\begin{exm} \label{no-mea-ball}
Consider the tms $y \in \mathscr{M}_{2,6}$
\footnote{Throughout the paper, the moments of a tms
are listed in the graded lexicographical ordering,
and moments of different degrees are separated by semicolons.}:
\[
(28;  0,  0;  1.1,  0,    3.4;  0,   0,   0,  0;
1.1,   0,  1.2,  0,  1.6;  0,   0,   0,    0,   0,
0;  28,  0,  3.4,  0,   1.6,   0,  1.2).
\]
Its $3$rd order moment matrix $M_3(y)$ is positive definite.
Let $K=\{x\in \re^2: \|x\|_2^2 \leq 25\}$.
Choose $\xi$ to be the tms in $\mathscr{M}_{2,6}$ induced by the probability measure
uniformly distributed on the unit ball.
For $k = 3, 4, 5$, we solve $\reff{max-lmd:K-qmod}$
(which is same as \reff{max-lmd:extn-y},
since $K$ is defined by a single inequality) using {\tt GloptiPoly} \cite{GloPol3},
and get optimal values $\lmd_k$ numerically as
\[
\lmd_3 \approx 0.3702 >0, \quad \lmd_4 \approx 0.0993 >0, \quad \lmd_5 \approx -0.2370 < 0.
\]
By Theorem~\ref{sdp-pr:K-mo} or \ref{put-infeas},
we know this tms does not admit a $K$-measure.
\qed
\end{exm}

By Theorem~\ref{sdp-pr:K-mo}, when $K$ is compact,
a tms $y$ belongs to $\mathscr{R}_d(K)$ if and only if
$\lmd_k$ is nonnegative for all $k$,
which is difficult to check in applications.
Similarly, it is also difficult to check $\hat{\lmd}_k \geq 0$ for all $k$.
However, the items (iv) of Theorems~\ref{sdp-pr:K-mo} and \ref{put-infeas}
provide a certificate for checking the membership $y\in \mathscr{R}_d(K)$.
This is because we can easily check the flat extension condition \reff{cond:FEC} as follows.
Let $(\lmd_, w^*)$ be an optimal pair for \reff{max-lmd:extn-y} or \reff{max-lmd:K-qmod}.
If a truncation $w^*|_t(t \geq d)$ is flat and $\lmd_k \geq 0$,
a $K$-measure representing $y$ can be constructed from the
relation $y = w^*|_d + \lmd_k \xi$,
because the tms $\xi$ admits a measure by its choice.
If we have $\mu \in meas(w^*,K)$ and $\nu \in meas(\xi,K)$,
then $\mu+\lmd_k \nu$ belongs to $meas(y,K)$.

\begin{exm} \label{exmp:(2,6)-10pt}
Consider the tms $y \in \mathscr{M}_{2,6}$:
\[
\baray{c}
            (1;\,
         7/50,\,
         1/50;\,
          2/5,\,
            0,\,
          2/5;\,
      91/1250,\,
       -6/625,\,
       42/625,\,
      37/1250;\, \\
   6973/25000,\,
     -42/3125,\,
   1777/25000,\,
      42/3125,\,
   6973/25000;\,
   1267/31250,\, \\
   -222/15625,\,
    504/15625,\,
     72/15625,\,
    546/15625,\,
    781/31250;\,
 23713/100000,\, \\
     -42/3125,\,
  2929/100000,\,
            0,\,
  2929/100000,\,
      42/3125,\,
 23713/100000).
\earay
\]
Let $K$ be the 2-dimensional unit ball; thus $d_g=1$.
The 3rd order moment matrix $M_3(y)$
is positive definite. Let $\xi$ be the tms in $\mathscr{M}_{2,6}$ induced by the uniform
probability measure $\nu$ on the unit ball.
We solve \reff{max-lmd:extn-y} by {\tt GloptiPoly} \cite{GloPol3}.
For $k=5$
we get the optimal value $\lmd_5 \approx 1.2\cdot 10^{-7}$.
The computed optimal $w^*$ is flat,
because $\rank M_{4}(w^*) = \rank M_{5}(w^*) = 10$.
\footnote{The ranks here are evaluated numerically.
We ignore singular values smaller than $10^{-6}$ when evaluating ranks.
The same procedure is applied
in computing ranks throughout this paper.}
By Theorem~\ref{thm-CF:fec},
we know $w^*$ admits a unique $10$-atomic measure $\mu$.
\qed
\end{exm}

For a flat tms, a numerical algorithm is given in \cite{HenLas05}
to find the support of its finitely atomic representing measure.
In Example~\ref{exmp:(2,6)-10pt},
by using this algorithm, we get the support of
the $10$-atomic measure $\mu$ admitted by $w^*$ there:
\[
\left\{ \pm (1,0), \quad \pm(0,1), \quad  (\pm 1/2, \pm 1/2),
\quad (4/5, -3/5), \quad (3/5, 4/5) \right\}.
\]
From $y = w^*|_6 + \lmd_5 \xi$, we know
$
\mu +\lmd_5 \cdot \nu
$
is a ball-measure representing the tms $y$ there.

\subsection{Finding representing measures via shifting}

The sign of $\lmd_\infty$ is critical in checking the membership in $\mathscr{R}_d(K)$.
When $\mbox{int}(K)\ne\emptyset$, by Lemma~\ref{lm:Rz:K>0},
a tms $y$ belongs to $\mbox{int}\left(\mathscr{R}_d(K)\right)$
if and only if its Riesz functional $\mathscr{L}_{y}$ is strictly $K$-positive.
By Theorem~\ref{thm:Tchak}, for compact $K$, we know $\mathscr{R}_d(K)$ is closed.
Thus, for compact $K$ with $\mbox{int}(K)\ne\emptyset$,
we can check the membership in $\mathscr{R}_d(K)$ as:
\bit
\item
If $\lmd_\infty > 0$,
$y$ lies in the interior of $\mathscr{R}_d(K)$;

\item
if $\lmd_\infty = 0$, $y$ lies on the boundary of $\mathscr{R}_d(K)$;

\item
if $\lmd_\infty < 0$, $y$ lies outside $\mathscr{R}_d(K)$.
\eit
When $\lmd_{\infty} > -\infty$,
the shifted tms $\hat{y}:=y-\lmd_\infty \xi$ always admits a $K$-measure.
When $\lmd_\infty \geq 0$, if we have $\mu \in meas(\hat{y},K)$ and $\nu \in meas(\xi,K)$,
then $\mu + \lmd_\infty \nu$ is a $K$-measure representing $y$,
because of the relation $y = \hat{y} + \lmd_\infty \xi$.
Thus, it is enough to investigate how one can get
a $K$-measure representing $\hat{y}$.

When $\mbox{int}(K)\ne\emptyset$, Theorem~\ref{sdp-pr:K-mo}(iii) implies that
$\lmd_\infty > -\infty$ and $\lmd_k = \lmd_\infty$
is achieved at some step $k$ if and only if
a measure representing $\hat{y}$ is $(K,d)$-semialgebraic.
We are mostly interested in this case,
since it is not possible to solve \reff{max-lmd:extn-y}
for infinitely many $k's$.
Therefore, in the following, we assume $\hat{y}$
admits a $(K,d)$-semialgebraic measure.

For an optimizer $w^*$ of \eqref{max-lmd:extn-y},
the kernel of the moment matrix $M_k(w^*)$ is a useful tool
in constructing supports of representing measures and their vanishing ideals.
In particular, the kernel $\ker M_k(w^*)$ is of the greatest interest
when $M_k(w^*)$ achieves the maximum rank
over the set of all optimizers of  \eqref{max-lmd:extn-y}.
This approach was introduced by
Lasserre, Laurent and Rostalski \cite{LLR08}
when they compute zero-dimensional real radical ideals
via semidefinite relaxations. For more details,
we refer to the surveys \cite{Lau} by Laurent
and \cite{LR12} by Laurent and Rostalski.
In the following, we use this approach.

For each $k$, denote by $\tt{Opt}_k(K)$ the set of optimizers of
\eqref{max-lmd:extn-y}.

\begin{theorem} \label{thm:suf-Kmom}
Let $K$ be defined in \reff{def:K} (not necessarily compact),
and $y,\xi,\hat{y},\lmd_k,\lmd_\infty$ be same as in Theorem~\ref{sdp-pr:K-mo}.
Suppose $\lmd_{k_0}=\lmd_\infty$ and there exist $\mu \in meas(\hat{y},K)$
and $0 \ne q \in \re[x]_d \cap Pr_{k_1}(K)$, for some $k_0, k_1$, such that
\be \label{mu-(K,d)salg}
\supp{\mu} \,\subseteq \,
V:= \{u\in K: \, q(u) = 0\}.
\ee
Let $k\geq \max(k_0,k_1)$. If $\tt{Opt}_k(K) \ne \emptyset$,
then for each $(\lmd,w) \in \tt{Opt}_k(K)$ we have:

\bit

\item [(i)] For $k$ sufficiently large, it holds that
$I(V) \subseteq \langle \ker M_k(w) \rangle$.
%
%

\item [(ii)] Let $r = \max \{\rank\,M_k(v): \, (\lmd,v) \in \tt{Opt}_k(K) \}$.
If $\rank\,M_k(w)=r$, then $\langle \ker M_k(w) \rangle \subseteq  I(\supp{\mu})$.

\item [(iii)] If $|V| < \infty$, then $w|_{2k-2}$ is flat
for $k$ sufficiently large.

\item [(iv)] If $|\supp{\mu}|=\infty$ and $\rank\,M_k(w)=r$,
then $w|_{2t}$ can not be flat for all $0<t\leq k$.

\eit
\end{theorem}

\begin{proof}
Since $q$ is in $Pr_{k_1}(K)$, there exist SOS polynomials $\sig_\nu$ such that
\[
q = \sum_{\nu \in \{0,1\}^m}  \sig_\nu g_\nu.
\]
Write each $\sig_\nu = \sum_j \theta_{\nu,j}^2$.
For every $k\geq \max(k_0,k_1)$, it holds that (note $w|_d = \hat{y}$)
\[
0 = \int_K q d \mu = \langle q, \hat{y} \rangle =
\sum_{ \nu \in \{0,1\}^m } \sum_j
\theta_{\nu,j}^T M_{k-d_\nu}(g_\nu \ast w)   \theta_{\nu,j}.
\]
Because $M_{k-d_\nu}(g_\nu \ast w)  \succeq 0$ for all $\nu$, we have
$M_{k-d_\nu}(g_\nu \ast w)  \theta_{\nu,j} = 0$ for all $\nu,j$.

(i) Let $\{h_1,\ldots, h_r\}$ be a Grobner basis
of the vanishing ideal $I(V)$, under a total degree ordering.
Then, each $h_i$ vanishes on $V$. By Theorem~\ref{RealNul}, for each $h_i$,
there exist an integer $\ell \geq 1$, a polynomial $\phi$
and SOS polynomials $\varphi_\nu$ such that
\be
\label{eq:nssq}
h_i^{2\ell} +  q\phi + \sum_{\nu \in \{0,1\}^m} \varphi_\nu g_\nu = 0.
\ee
Write $\phi = \phi_1^2 - \phi_2^2$ for some $\phi_1,\phi_2 \in \re[x]$.
From $\theta_{\nu,j} \in \ker M_{d-d_\nu}(g_\nu \ast w)$,
by Lemma~\ref{lm:pq-ker}, we must have
$\theta_{\nu,j}\phi_s \in \ker M_{d-d_\nu}(g_\nu \ast w)$ if
\[
\deg(\theta_{\nu,j}\phi_s) + \lceil \deg(g_\nu)/2 \rceil \leq k-1 \, (s=1,2).
\]
This implies that
\begin{align*}
\langle q \phi, w\rangle = & \sum_{ \nu \in \{0,1\}^m, j}
\Big( (\theta_{\nu,j}\phi_1)^T M_{k-d_\nu}(g_\nu \ast w) (\theta_{\nu,j}\phi_1) \\
& \quad - (\theta_{\nu,j}\phi_2)^T M_{k-d_\nu}(g_\nu \ast w) (\theta_{\nu,j}\phi_2)
\Big)=0.
\end{align*}
Combined with \eqref{eq:nssq}, this gives
\[
\langle M_k(w) h_i^\ell, h_i^\ell \rangle
+ \sum_{\nu \in \{0,1\}^m} \langle \varphi_\nu g_\nu, w \rangle =0.
\]
Since $M_{k-d_\nu}(g_\nu \ast w) \succeq 0$ for all $\nu$,
we know $\langle \varphi_\nu g_\nu, w \rangle \geq 0$ for all $\nu$,
which implies $M_k(w) h_i^\ell = 0$.
By an induction on $\ell$, we can get $h_i \in \ker M_k(w)$
(cf. Lemma~3.9 of \cite{LLR08}).
So, $I(V) \subseteq \langle \ker M_k(w) \rangle $ for $k$ big enough.

(ii) Let $v \in \mathscr{M}_{n,2k}$ be the tms induced by $\mu$,
i.e., $v=\int [x]_{2k} d \mu$.
Then $(\lmd_k, v)$ belongs to $\tt{Opt}_k(K)$.
The $\rank\, M_k(w)$ being maximum over $\tt{Opt}_k(K)$ implies
\[
\rank\, M_k(w)   \geq   \rank\, M_k((w+v)/2 )
= \rank\, M_k( w+v ).
\]
Since both $M_k( w) \succeq 0$ and $M_k(v) \succeq 0$, it holds that
\[
\ker M_k( w+v )  = \ker M_k(w) \cap \ker M_k(v).
\]
The relation $\rank\,M_k(w)\geq \rank\,M_k(w+v)$
and the above imply
\[
\ker M_k(w) =  \ker M_k(w+v) \subseteq \ker M_k(v).
\]
For all $ f \in \ker M_k(w)$, we have $ f \in \ker M_k(v)$ and
\[
0 = f^T M_k(v) f = \int_K f^2 d\mu.
\]
Thus, $f$ vanishes on $\supp{\mu}$, and
$\langle \ker M_k(w) \rangle \subseteq I(\supp{\mu})$.

(iii)
When $|V| < \infty$, the quotient space $\re[x]/I(V)$ is finitely dimensional.
Let $\{b_1, \ldots, b_t\}$ be a standard basis for it.
Then, every monomial $x^\af$ can be written as
\[
x^\af = r(\af) + \sum \phi_i h_i, \quad  \deg(\phi_i h_i) \leq |\af|, \quad
r(\af) \in span\{b_1,\ldots, b_t\}.
\]
If $k$ is big enough, we have $h_i \in \ker M_k(w)$ for all $i$, by item (i).
We also have $\phi_ih_i \in \ker M_k(w)$ for all $i$
if $|\af| \leq k-1$, by Lemma~\ref{lm:pq-ker}.
Thus, $x^\af - r(\af)$ belongs to $\ker M_k(w)$ for all $|\af|=k-1$. So, if
\[
k-1 - d_g > \max\{ \deg(b_1),\ldots, \deg(b_t)\}
\]
is big enough, every $\af$-th ($|\af|=k-1$) column of $M_k(w)$ is a linear combination
of the $\bt$-columns of $M_k(w)$ with $|\bt|<k-1-d_g$.
This means that the tms generating the moment matrix $M_{k-1}(w)$ is flat,
i.e., the truncation $w|_{2k-2}$ is flat.

(iv) For a contradiction, suppose $w|_{2t}$ is flat for some $0<t\leq k$.
Then, for every $|\af|=t$, there exists a polynomial
$\varphi_\af \in \re[x]_{t-1}$
such that $x^\af - \varphi_\af$ belongs to $\ker M_t(w)$. This implies that
\[
0 = (x^\af - \varphi_\af)^T M_t(w)(x^\af - \varphi_\af) =
(x^\af - \varphi_\af)^T M_k(w)(x^\af - \varphi_\af).
\]
Since $M_k(w)\succeq 0$, we have  $x^\af - \varphi_\af \in \ker M_k(w)$.
By item (ii), we have $x^\af - \varphi_\af \in I(\supp{\mu})$.
Then a simple induction on $|\af|$ shows that
\[
x^\af \equiv \psi_\af\, \mod \, I(\supp{\mu})
\]
for some $\psi_\af \in \re[x]_{t-1}$, whenever $|\af| \geq t$.
That is, every high degree polynomial is equivalent to
a polynomial of degree not bigger than $t-1$ modulo $I(\supp{\mu})$.
So, $I(\supp{\mu})$ is zero-dimensional,
i.e., $\dim \re[x]/I(\supp{\mu}) < \infty$,
which contradicts $|\supp{\mu}|=\infty$
(cf. Proposition~2.1 of Sturmfels \cite{Stu02}).
\end{proof}

\noindent {\it Remarks:}
a) In Theorem~\ref{thm:suf-Kmom}, we do not need to assume $K$ is compact.
b) The condition \reff{mu-(K,d)salg} implies $\mathscr{L}_{\hat{y}}(q)=0$
and the Riesz functional $\mathscr{L}_{\hat{y}}$ is not strictly $K$-positive.
Thus, the tms $\hat y$ lies on the boundary of $\mathscr{R}_d(K)$.
This is an advantage of using $\hat y$ instead of $y$.
c) When \reff{mu-(K,d)salg} holds and $|V|<\infty$,
the truncation $w^*|_{2k-2}$ is flat for $k$ big enough,
but $w^*$ itself might not be flat.

\subsection{Tms with a finite $K$-variety}

In this subsection, we consider tms' whose associated algebraic varieties
have finite intersection with $K$.
Let $d_0 = \lfloor d/2\rfloor$. The \df{algebraic variety associated to}
a tms $y \in \mathscr{M}_{n,d}$ is defined as (cf. \cite{FiNi1})
\be \label{tms-var:y}
\mathcal{V}(y) := \bigcap_{p\in \ker M_{d_0}(y) } \{x\in \re^n: p(x) =0 \}.
\ee
Then, we define the \df{$K$-variety associated to $y$} as
\be \label{K-var:y}
\mathcal{V}_K(y) \, := \, \mathcal{V}(y) \cap K.
\ee
When the moment matrix $M_{d_0}(y)$ is singular
(thus the variety $\mathcal{V}(y)$ is a proper subset of $\re^n$,
like a finite set or a curve, etc.),
there exists work on discussing whether $y$ admits a measure.
We refer to \cite{CF02,CF052,CF11,F082}.

Clearly, for $\mathcal{V}_K(y)$ to be a proper
subset of $K$, the moment matrix $M_{d_0}(y)$ must be singular,
and any $\lmd>0$ is not feasible in \reff{max-lmd:extn-y}.
To see this, suppose otherwise a pair $(\lmd, w)$ with $\lmd>0$
is feasible in \reff{max-lmd:extn-y}.
When $\mathcal{V}_K(y) \ne K$,
there must exist $p \in \ker M_{d_0}(y)$ such that $p \not\equiv 0$ on $K$.
The relation $w|_d = y - \lmd \xi$ implies
\[
M_{d_0}(y) = M_{d_0}(w) + \lmd M_{d_0}(\xi).
\]
The above moment matrices are all positive semidefinite.
Thus, from $p \in \ker M_{d_0}(y)$ and $\lmd>0$, we know
$p \in \ker M_{d_0}(\xi)$ and $\mathscr{L}_{\xi}(p^2) = 0$.
This contradicts the strict $K$-positivity of $\xi$.
Therefore, when $\mathcal{V}_K(y) \ne K$,
the optimal value $\lmd_k$ of \reff{max-lmd:extn-y}
must be zero if $y$ belongs to $\mathscr{R}_d(K)$,
and \reff{max-lmd:extn-y} is equivalent to the problem
\be  \label{fin-Kvar:y}
\mbox{ find } \quad  w \in \mathscr{M}_{n,2k}  \quad
s.t. \quad w|_d =y,\, M_{k-d_\nu}(g_\nu \ast w) \succeq 0 \, \forall \nu \in \{0,1\}^m.
\ee

\begin{theorem}  \label{thm:K-finite}
Let $K$ be defined in \reff{def:K} (not necessarily compact),
and $y$ be a tms in $\mathscr{M}_{n,d}$. Suppose $y$ belongs to $\mathscr{R}_d(K)$
and $|\mathcal{V}_K(y)| < \infty$. If $k$ is sufficiently large,
then every $w$ satisfying \reff{fin-Kvar:y}
has a flat truncation $w|_{2k-2}$.
\end{theorem}
\begin{proof}
As we have seen in the above,
every optimal $\lmd_k$ in \reff{max-lmd:extn-y} is zero.
Thus, $\lmd_\infty=0$ and $\hat{y}=y-\lmd_\infty \xi = y$.
Choose a measure $\mu \in meas(y,K)$ and a basis $\{ p_1,\ldots, p_r \}$
of $\ker M_{d_0}(y)$.
Let $q=p_1^2+\cdots+p_r^2 \in \re[x]_d$. From
\[
\int_K q d \mu = \sum_{i=1}^r \int_K p_i^2 d \mu = \sum_{i=1}^r p_i^TM_{d_0}(y)p_i = 0,
\]
we know $\supp{\mu} \subseteq \mc{Z}(q) \cap K$.
The finiteness of $\mathcal{V}_K(y)$ implies the set $\mc{Z}(q) \cap K$
is finite. Then, this theorem
follows from item (iii) of Theorem~\ref{thm:suf-Kmom}.
\end{proof}

\noindent {\it Remark:}
As one can see in the proof, Theorem~\ref{thm:K-finite}
is an application of Theorem~\ref{thm:suf-Kmom}.
If the moment matrix $M_{d_0}(y)$ is singular and
the tms $y$ belongs to $\mathscr{R}_d(K)$ and $\mathcal{V}_K(y)\ne K$,
we have already seen that the optimal value of \reff{max-lmd:K-qmod} must be zero,
and hence \reff{max-lmd:K-qmod}
is equivalent to the feasibility problem \reff{fin-Kvar:y}
(see the argument preceding Theorem~\ref{thm:K-finite}).

%
%

\begin{exm}  \label{emp:fi-Kvar(2,4)}
Consider the tms $y \in \mathscr{M}_{2,4}$:
\[
( 1;\,   0,\,   0;\,  1,\,  0,\,  1 ;\,
0 ,\,  0,\,  0,\,   0;\,   1,\, 0,\,  1,\, 0,\,  1).
\]
The set $K=\re^2$, $d_0=2$, and the moment matrix $M_{d_0}(y)$ is singular
(it has rank $4$ and size $6\times 6$).
The $K$-variety is defined by $x_1^2=x_2^2=1$.
For $k=4$, we solve \reff{fin-Kvar:y} and get a feasible tms $\hat{w} \in \mathscr{M}_{2,8}$.
Its truncation $\hat{w}|_6$ is flat (while $\hat{w}$ itself is not),
and it, as well as $y$, admits a $4$-atomic measure supported on $\{(\pm 1, \pm 1)\}$.
\qed
\end{exm}

\subsection{The case $K=\re^n$}

When $K=\re^n$ is the whole space,
we can similarly apply the approach described in the preceding subsections.
However, since $\re^n$ is not compact,
weaker conclusions can be made.

Let $y$ be a tms in $\mathscr{M}_{n,d}$ and $d_0=\lfloor d/2 \rfloor$.
If $y$ is extendable to a flat tms $w\in\mathscr{M}_{n,2k}$,
then we can get a finitely atomic representing measure for $y$, 
by Theorem~\ref{thm-CF:fec}. In analogue to \reff{max-lmd:extn-y},
for an integer $k\geq d_0$, consider the optimization problem
\be  \label{mom-R^n:w>=y}
\left\{
\baray{rl}
\underset{ \eta \in \re, \, w\in\mathscr{M}_{n,2k} }{\max} &  \eta \\
s.t. & w|_d =y-\eta \xi, \ \  M_k(w) \succeq 0.
\earay\right.
\ee
Similarly, we choose $\xi$ to be a tms in 
$\mathscr{R}_d(\re^n)$ with $M_{d_0}(\xi) \succ 0$.
Clearly, if $\eta$ is feasible in \reff{mom-R^n:w>=y}, then
\be \label{eta-star}
\eta \leq \eta^*:= \lmd_{min}
\left( M_{d_0}(\xi)^{-1/2} M_{d_0}(y) M_{d_0}(\xi)^{-1/2}\right).
\ee
When $\eta < \eta^*$, the moment matrix $M_{d_0}(y-\eta \xi)$ is positive definite,
and the tms $y-\eta \xi$ is extendable to a tms $w \in\mathscr{M}_{n,2k}$ with $M_k(w) \succ 0$.
So the optimal value of \reff{mom-R^n:w>=y} is equal to $\eta^*$ for every $k$.
Clearly, we have $\eta^*\geq 0$ if and only if
the matrix $M_{d_0}(y)$ is positive semidefinite,
and $\eta^* = 0$ if and only if $\lmd_{min}(M_{d_0}(y)) = 0$.

\begin{theorem} \label{thm:tms-R^n}
Let $y$ be a tms in $\mathscr{M}_{n,d}$ and $\xi$ be a tms in
$\mathscr{R}_d(\re^n)$ with $M_{d_0}(\xi) \succ 0$.
Suppose $(\eta, w)$ is feasible for \reff{mom-R^n:w>=y}.
Let $\hat{y}:=y-\eta^* \xi$ and $k\geq d_0$.

\bit

\item [(i)] If $\eta \geq 0$ and $w|_t(t\geq d)$ is flat,
then $y$ belongs to $\mathscr{R}_d(\re^n)$.

\item [(ii)] If $\mu$ is a measure in $meas(\hat{y}, \re^n)$,
then there exists  $p\in \Sig_{n,2d_0}$ such that
$
\supp{\mu} \subseteq \mc{Z}(p).
$
\eit
In the following, suppose $\mu, p$ satisfy (ii)
and $(\eta, w)$ is optimal for \reff{mom-R^n:w>=y}.
\bit
\item [(iii)]
If $|\mc{Z}(p)|<\infty$, then $w|_{2k-2}$ is flat for $k$ sufficiently large.

\item [(iv)]
If $|\supp{\mu}| = +\infty$ and $\rank\,M_k(w)$ is maximum,
then $w|_{2t}$ can not be flat for all $0<t\leq k$.

\eit
\end{theorem}
\begin{proof}
(i) is implied by  $y = \eta \xi + w|_d$ and Theorem~\ref{thm-CF:fec}.

(ii) The dual optimization problem of \reff{mom-R^n:w>=y} is
\be \label{R^n:dual-sos}
\underset{f}{\min} \quad \langle f, y \rangle  \quad
s.t. \quad   \langle f, \xi \rangle = 1, \, f \in \Sig_{n,2d_0}.
\ee
Since \reff{mom-R^n:w>=y} has a feasible $w$ with $M_k(w) \succ 0$,
the optimization problem \reff{R^n:dual-sos} has a minimizer $p\in \Sig_{n,2d_0}$,
and its optimal value equals $\eta^*$. So
\[
\int_{\re^n} p d\,\mu = \langle p, \hat{y} \rangle =
\langle p, y \rangle - \eta^* = 0.
\]
The nonnegativity of $p$ in $\re^n$ implies $\supp{\mu} \subseteq \mc{Z}(p)$.

(iii) It can be proved in a way similar to items (iii)-(iv)
of Theorem~\ref{thm:suf-Kmom}. Write $p = p_1^2 + \cdots + p_L^2$.
Then $\supp{\mu} \subseteq \mc{Z}(p)$ implies
\[
\sum_{i=1}^L p_i^TM_{d_0}(\hat{y}) p_i = \langle p, \hat{y} \rangle
= \int_{\re^n} p d \mu = 0.
\]
Let $\{h_1,\ldots, h_r\}$ be a Grobner basis
for the vanishing ideal $I(\mc{Z}(p))$, under a total degree ordering.
Similarly, by Theorem~\ref{RealNul}, for each $h_i$,
there exist $\ell \geq 1$ and $f_1,\ldots, f_r\in \re[x]$ satisfying
$
h_i^{2\ell} + f_1 p_1 + \cdots + f_r p_r = 0.
$
As in the proof of Theorem~\ref{thm:suf-Kmom},
we can similarly prove $h_i \in \ker M_k(w)$ and $\langle f_jp_j, w \rangle = 0$
for $j=1,\ldots,r$, when $k$ is sufficiently large.
When $\mc{Z}(p)$ is finite, we can similarly prove
$w|_{2k-2}$ is flat as for (iii) of Theorem~\ref{thm:suf-Kmom}.

(iv)
The proof is same as for (iv) of Theorem~\ref{thm:suf-Kmom}.
\end{proof}

\begin{exm}
Consider the same tms $y$ as in Example~\ref{emp:fi-Kvar(2,4)}.
Take $K= {\mathbb R}^2$; thus $d_g=1$ and $d_0=2$.
The moment matrix $M_1(y)$ has rank $3$ and $M_2(y)$ has rank $4$,
so $y$ itself is not flat. The $6\times 6$ moment matrix  $M_2(y)$ is singular,
so  $\eta^*=0$.
Choose $\xi$ in \reff{mom-R^n:w>=y} to be the tms in $\mathscr{M}_{2,4}$
generated by the standard Gaussian measure with mass one.
For $k=4$, solve the optimization problem \reff{mom-R^n:w>=y}.
The computed optimal tms $w^* \in \mathscr{M}_{2,8}$ is not flat, but its truncation
$w^*|_6$ is flat and admits a unique $4$-atomic measure
with support $\{(\pm 1, \pm 1)\}$.
The tms $y$ admits the same measure.
\qed
\end{exm}

Unlike for the case that $K$ is compact, the shifted tms $\hat{y}:=y-\eta^* \xi$
does not always admit a measure, even when $n=1$. For instance, for the tms
$(1,1,1,1,2)\in \mathscr{M}_{1,4}$,
its moment matrix $M_2(y)$ is positive semidefinite but singular. So, $\eta^*=0$ but
$\hat{y}=y$ does not admit a measure
(cf. \cite[Example~2.1]{CF09}).

\section{A Practical Method For Finding Flat Extensions}
\label{sec:heuristic}
\setcounter{equation}{0}

The preceding section discusses how to check
if a tms $y$ admits a $K$-measure.
The limit $\lmd_{\infty}$ of the optimal values of
\reff{max-lmd:extn-y} or \reff{max-lmd:K-qmod}
plays a critical role. For a compact set $K$,
if $\lmd_{\infty} \geq 0$, then $y$ admits a $K$-measure;
otherwise, it does not.
When $\lmd_{\infty} \geq 0$, a representing measure for $y$
is also constructible from the relation
$y = \hat{y} + \lmd_\infty \xi$ if $\hat{y}$ admits a $K$-measure.
In the case $\lmd_{\infty} > 0$, this approach typically
does not give a finitely atomic representing measure,
because $\xi$ is usually generated by a measure with infinite support.

In many applications, one is interested in getting a finitely atomic representing measure.
By Theorem~\ref{CF:flat-extn} of Curto and Fialkow,
a tms $y$ admits a $K$-measure if and only if
it is extendable to a flat tms $w$ of higher degree.
This means that if $y$ admits a $K$-measure, then a finitely atomic measure
for $y$ can be obtained by finding a flat extension.
In some special cases (e.g., a moment matrix is singular,
or its associated variety is finite, etc),
there exists work investigating when and how a flat extension
could be found (cf. \cite{CF02,CF052,CF11,F082}).
For the general case, especially when a moment matrix is positive definite,
methods for constructing a flat extension and determining if it exists
are relatively unexplored.

This section presents a practical SDP method for this purpose,
alluded to in \S \ref{sec:contrib}(2).
Our numerical experiments show that
it often finds a flat extension of $y$ when it exists.

\subsection{TKMP}

Suppose a tms $w \in \mathscr{M}_{n,2k}$ ($k \geq d/2$) is an extension of $y$.
Clearly, if a truncation $w|_t (t\geq d)$ is flat,
then the finitely atomic measure admitted by $w|_t$ 
is a representing measure for $y$.
To find such an extension, for an integer $k\geq d/2$,
consider the semidefinite optimization problem:
\be \label{min-tr:qmod}
\left\{
\baray{rl}
\underset{ w\in\mathscr{M}_{n,2k} }{\min} &   c_k^Tw   \\
s.t. & w|_d =y, \, M_{k-1}(\rho \ast w)  \succeq 0, \\
& M_{k-d_i}(g_i \ast w) \succeq 0, \, i =0, 1,\ldots,m.
\earay \right.
\ee
In the above, $c_k$ is a generic vector, and the polynomial
$\rho(x)=R^2-\|x\|_2^2$ is such that $K\subseteq B(0,R)$.
Recall that $g_0 = 1$ and $M_k(w) = M_{k-d_0}(g_0 \ast w)$.
\bnum
\item
If $w^*$ is an optimizer of \reff{min-tr:qmod} and $w^*|_t (t\geq d)$ is flat,
then the finitely atomic measure admitted by $w^*|_t$ also represents $y$.

\item
If $y$ does not admit a $K$-measure,
then \reff{min-tr:qmod} will not be feasible
for $k$ sufficiently large (see the remark after Theorem~\ref{put-infeas}).

\enum
We would like to remark that the optimization problem \reff{min-tr:qmod},
as well as \reff{fin-Kvar:y},
is an obvious and natural way to check Theorem~\ref{CF:flat-extn} of Curto and Fialkow.

We show some examples of illustrating this method.

\begin{exm} \label{co-pos:n=4}
Let $K=\left\{x\in\re_+^4: \sum_{i=1}^4 x_i \leq 1\right\}$.
Consider the tms $y \in \mathscr{M}_{4,2}$:
\[
\baray{c}
(1;\, 1/5,\, 1/5,\, 1/5,\, 1/5;\, 1/10,\, 1/20,\,
0,\, 0,\, 1/10,\, 1/20,\, 0,\, 1/10,\, 1/20,\, 1/10).
\earay
\]
Its moment matrix $M_1(y)$ is positive definite.
Choose $c_k$ to be the vector of all ones.
Solve \reff{min-tr:qmod} for $k=4$,
and get an optimizer $w^*$. The tms $w^*$ and its
truncations $w^*|_4, w^*|_6$ are all flat.
They admit a $5$-atomic measure supported on the points:
\[
(0.5,  0, 0, 0), \quad (0,0,0, 0.5),  \quad (0, 0.5, 0.5, 0),
\quad (0,  0,    0.5,  0.5), \quad (0.5,  0.5, 0,0).
\]
Since $w^*$ is an extension of $y$,
the tms $y$ admits the same $5$-atomic measure.
\qed
\end{exm}

Next we describe an experimental test
of using \reff{min-tr:qmod} to find flat extensions
for randomly generated tms'.
It produced a flat extension for all the instances we tested.

\begin{exm}[random problems] \label{rand:Ball}
Let $K=[-1,1]^n$ be the unit hypercube.
We generate testing instances as follows.
In \reff{min-tr:qmod}, choose $c_k$ to be a Gaussian random vector,
and $(n,d)$ from Table~\ref{tab-Ball:(n,d)}.
{\small{
\begin{table}
\centering
\btab[htb]{|l|l|l|l|l|l|l|} \hline
n &  d   &  d   &  d   &  d    & d    \\ \hline
2 &  4(100)  &  6(100) &  8(50) &  10(50) & 12(50) \\ \hline
3 &  4(100)  &  5(100) &  6(50) &  7(50) & 8(50)  \\ \hline
4 &  3(100)  &  4(100) &  5(50) &  6(20) & 7(10)  \\ \hline
\etab.
\caption{Inside the parentheses
are the numbers of randomly generated instances we tested
for each pair $(n,d)$.}
\label{tab-Ball:(n,d)}
\end{table}
}}
By the theorem of Bayer and Teichmann \cite{BT},
a tms $y$ admits a $K$-measure if and only if
$y$ admits a $N$-atomic $K$-measure with $N:=\binom{n+d}{d}$.
So, we randomly generate $N$ points $u_1,\ldots, u_N$ from $[-1,1]^n$,
and set $y = a_1 [u_1]_d + \cdots a_N [u_N]_d$
for random positive numbers $a_i>0$.
In each case, we solve \reff{min-tr:qmod} starting with $k=d$.
If an optimal $w^*$ of \reff{min-tr:qmod} has a flat truncation $w^*|_t(t\geq d)$,
we stop; otherwise, increase $k$ by one, and then solve \reff{min-tr:qmod} again.
Repeat this process until we get a flat extension of $y$.
For each $(n,d)$, the number of randomly generated instances
is listed in the parenthesis after $d$ in Table~\ref{tab-Ball:(n,d)}.
In all the tested instances, we got a flat truncation $w^*|_t$, for some $k,t\geq d$.
Furthermore, the supports of the finitely atomic measures we obtained
often have cardinalities smaller than $N$
(we are not sure whether they are minimum or not).
\qed
\end{exm}

Our conclusion is that solving \reff{min-tr:qmod} is a practical method
for finding a flat extension of a tms $y$ when it admits a $K$-measure.

\subsection{A theoretical analysis}
\label{sec:heurThy}

While we do not understand well why solving \reff{min-tr:qmod}
always produced a flat extension in our numerical experiments.
Here, we present a bit of theoretical analysis
suggesting that this will often produce a flat extension.
Our analysis is to set up a linear-convex optimization problem
for which \reff{min-tr:qmod} is an approximation.
Then, under some  assumptions, we prove that
\reff{min-tr:qmod} asymptotically produces a flat extension.
We start with some definitions. Denote
\begin{align*}
E_{k}(y)  &:= \big\{ w \in \mathscr{M}_{n,2k}: \, w
\mbox{ is feasible for }  \reff{min-tr:qmod}
\big\}, \\
E_{\infty}(y) &:=\big\{ z \in \mathscr{R}_{\infty}(K): \, z|_d = y \big\}.
\end{align*}
Embed $E_k(y)$ into $\mathscr{M}_{n,\infty}$
by the mapping $w \mapsto (w, 0, \ldots )$ of adding zeros.
Thus, $E_{k}(y)$ and $E_{\infty}(y)$ can be thought of as
convex subsets of $\mathscr{M}_{n,\infty}$.
For every $w \in \mathscr{M}_{n,\infty}$, define $\|w\|_2$ as
$
\|w\|_2^2 = \sum_{\af \in \N^n} w_\af^2.
$
Define the Hilbert space
\be \label{Minf^2-Ban}
\mathscr{M}_{n,\infty}^2 := \{w \in \mathscr{M}_{n,\infty}: \|w\|_2 < \infty\}.
\ee
Up to a shifting and scaling, we can generally assume $K \subseteq B(0,R)$ with $R<1$.

We assume the vectors $c_k$ in \reff{min-tr:qmod} for different $k$
are consistent, i.e., each $c_{k-1}$ is a leading subvector of $c_k$.
Write $c_k(x) = c_k^T[x]_{2k}$.
Then, its limit is a real analytic function, which is denoted as
$
c(x) = \Sig_{\af } c_\af x^\af.
$
The analytic function $c(x)$ can also be thought of 
as a vector $c \in \mathscr{M}_{n,\infty}$. If $\|c\|_2 < \infty$,
then $c$ is a continuous linear functional acting on $\mathscr{M}_{n,\infty}^2$ as
\[
\langle c, w \rangle := \sum_\af c_\af w_\af.
\]
Now we consider the optimization problem:
\be \label{min-CM:m(y)}
\min_\mu  \quad \int_K c(x) d \mu
\quad \mbox{ s.t. }  \quad \mu \in meas(y,K).
\ee
When $y$ belongs to $\mathscr{R}_d(K)$, its feasible set is nonempty.
The objective is a linear functional acting on the convex set $meas(y,K)$.
Note that a vector $w$ belongs to $E_\infty(y)$ if and only if
it admits a measure in $meas(y,K)$. Thus,
the optimization problem \reff{min-CM:m(y)} is also equivalent to
\be \label{min<c,w>:E-inf}
\min_{w}  \quad \langle c, w \rangle
\quad \mbox{ s.t. }  \quad w \in E_{\infty}(y).
\ee
If the optimization problem \reff{min-CM:m(y)} 
has a unique minimizer $\mu^*$, then $\mu^*$
must have finite support (see Appendix).
Theorem~\ref{thm:cvg-flat} below shows that if the biggest support
of minimizing measures of \reff{min-CM:m(y)} is finite,
then \reff{min-tr:qmod} asymptotically yields a flat extension
as $k$ goes to infinity.
Theorem~A.2 in Appendix shows that for
a dense subset of analytic functions $c$ in $\mathscr{C}(K)$
(the space of all continuous functions defined on $K$),
this condition (see \reff{hyp:fin-sup}) holds.

\begin{theorem} \label{thm:cvg-flat}
Assume $K\subseteq B(0,R)$ with $R<1$, the $c_k's$ are consistent,
and $c$ is in $\mathscr{M}_{n,\infty}^2$.
Let {\tt Sol} be the set of all optimizers of \reff{min-CM:m(y)}.
Suppose there exists $N>0$ such that
\be \label{hyp:fin-sup}
\Big| \bigcup_{ \mu^* \in {\tt Sol} } \{\supp{\mu^*} \} \Big| \leq N.
\ee
Let $w^{(k)}$ be an optimizer of \reff{min-tr:qmod}.
Then, for every $r > d_g N $ ($d_g$ is given by \reff{df:dg})
and every $\mu^* \in {\tt Sol}$,
the tms $z^*:=\int_K [x]_{2r} d \mu^*$ is flat and
\be \label{dist:Sol}
\lim_{k\to\infty} dist( w^{(k)}|_{2r}, U ) = 0
\quad \mbox{ where } \quad
U = \left\{ \int_K [x]_{2r} d \mu^*: \, \mu^*\in {\tt Sol} \right\}.
\ee
In particular, if {\tt Sol} is a singleton,
then $\{w^{(k)}|_{2r}\}$ converges to a flat tms.

\end{theorem}
\begin{proof}
Let $r_0:=d_g N$.
If $z^*_0 = 0$, then $z^*$ must be a zero vector, and it is clearly flat.
So, we consider the general case $z^*_0 > 0$.
We claim that a truncation $z^*|_{l}$ of $z^*$ must be flat for some $0<l\leq r_0$.
Otherwise, if not, then we must have
\[
1 = \rank \, M_0(z^*) < \rank \, M_{d_g}(z^*) < \rank \, M_{2d_g}(z^*) < \cdots <
\rank \, M_{r_0}(z^*).
\]
The above then implies the contradiction:
\[
|\supp{\mu^*}| \geq \rank M_{r_0}(z^*) \geq  1 + r_0/d_g > N.
\]
So, $z^*|_{l}$ is flat for some $l\leq r_0$.
By Theorem~\ref{thm-CF:fec}, we know $\mu^*$ is
the unique representing measure of $z^*|_{l}$.
Since $z^*$ is an extension of $z^*|_{l}$ and admits the same measure $\mu^*$,
we know that $z^*$ must also be flat.

We prove \reff{dist:Sol} by contradiction.
Suppose otherwise it is false, then
\be \label{distU>=a}
dist( w^{(k)}|_{2r}, U ) \geq a \quad \mbox{ for some } a>0.
\ee
By Lemma~\ref{Ek-Bd:cw-tail}, we know the sequence
$\{ w^{(k)} \}$ is bounded in the Hilbert space $\mathscr{M}_{n,\infty}^2$.
By the Alaoglu Theorem (cf. \cite[Theorem~V.3.1]{Conway}),
$\{ w^{(k)} \}$ has a subsequence that is convergent in the weak-$\ast$ topology.
That is, it has a subsequence, which we denote again by $\{ w^{(k)} \}$ for convenience,
such that for some $w^* \in \mathscr{M}_{n,\infty}^2$
\[
\ell (w^{(k)}) \to \ell( w^*)
\]
for every continuous linear functional in $\mathscr{M}_{n,\infty}^2$
(the space $\mathscr{M}_{n,\infty}^2$ is self-dual).
This implies
(e.g., choose $\ell$ as $\ell(w)= w_\af$ for each $\af \in \N^n$)
that for all $t$
\be \label{2k2t->w*2t}
w^{(k)}|_{2t} \to  w^*|_{2t}, \quad \mbox{ and } \quad w^*|_d = y.
\ee
Since $w^{(k)}$ is feasible in \reff{min-tr:qmod}, from the above we get that
\[
M_{t-d_i}(g_i \ast w^*) \succeq 0 \,(i=0,1,\ldots,m), \quad
M_{t-1}(\rho \ast w^*)  \succeq 0.
\]
Hence, the vector $w^*$ is a full moment sequence
whose localizing matrices of all orders are positive semidefinite
(because we have $w^*|_{2t} \in Q_k(K)^*$ for all $t$).
By Lemma~3.2 of Putinar \cite{Put}, we know $w^*$ admits a $K$-measure
and $w^*$ belongs to $E_\infty(y)$. Note that $c$ is
a continuous linear functional acting on $\mathscr{M}_{n,\infty}^2$
as $\langle c, w \rangle$.  So,
\[
\langle c, w^* \rangle = \lim_{k \to \infty} \langle c, w^{(k)} \rangle.
\]
By Lemma~\ref{optval:cvg}, we know that $w^*$ is also an optimizer of \reff{min<c,w>:E-inf}
and $w^*|_{2r}$ belongs to $U$. However, \reff{2k2t->w*2t} implies the convergence
$
w^{(k)}|_{2r} \to  w^*|_{2r},
$
and \reff{distU>=a} implies that $w^*|_{2r}$ does not belong to $U$,
a contradiction. Therefore, \reff{dist:Sol} must be true.

When {\tt Sol} is a singleton, $\{w^{(k)}|_{2r}\}$ clearly converges to
a flat tms.
\end{proof}

By Theorem~\ref{thm:cvg-flat}, we know that every limit point of
the sequence $\{w^{(k)}|_{2r}\}$ is flat.
We would like to remark that the conditions in
Theorem~\ref{thm:cvg-flat} are only for theoretical analysis.
In practical implementations, they are not required.

\subsection{The case $K=\re^n$}

When $K=\re^n$, we propose a similar method like \reff{min-tr:qmod} for solving TMPs.
To find a finitely atomic representing measure for a tms,
it is enough to find one of its flat extensions.
Like \reff{min-tr:qmod}, a practical method is that
for an order $k\geq d/2$ we solve the SDP problem:
\be \label{tmp-R:minTr}
\left\{
\baray{rl}
\underset{  w\in\mathscr{M}_{n,2k} }{\min} &  c_k^Tw  \\
s.t. & w|_d =y, \,  M_k(w) \succeq 0.
\earay \right.
\ee
For $k > d/2$, the above feasible set is typically unbounded.
To guarantee \reff{tmp-R:minTr} has a minimizer,
we usually choose $c_k$ as $c_k^T [x]_{2k} = [x]_k^T C_k [x]_k$
with $C_k \succeq 0$.
A simple choice for $C_k$ is the identity matrix.

\begin{exm} \label{emp:(2,4)6supt}
Consider the tms $y \in \mathscr{M}_{2,4}$:
\[
( 1;\,    -1/6,\,     1/2;\,    3/2,\,     0,\,     3/2;\,
-7/6,\,     1,\,     1/3,\,     3/2;\,    7/2,\,
-1,\,    2,\,     1,\,    7/2).
\]
Its moment matrix $M_2(y)$ is positive definite.
By Theorem~3.3 of \cite{FiNi1}, this tms $y$ admits a measure.
Choose $c_k$ as $c_k^T[x]_{2k} = [x]_k^T[x]_k$,
and solve \reff{tmp-R:minTr} for $k=4$.
The computed optimal $w^*\in \mathscr{M}_{2,8}$ is flat.
It admits a $7$-atomic measure supported on the points
\footnote{Here only four decimal digits are shown to present points of the support.}:
\[
\baray{c}
(-1.1902,   -1.4545),
(-0.2725,   -1.4977),
( 1.4406,  -0.9396),
( 0.2150,   -0.4613), \\
(-1.7147,    0.1952),
(-1.6942,    1.1593),
( 0.9645,    1.7098).
\earay
\]
The tms $y$ admits the same $7$-atomic measure.
\qed
\end{exm}

We would like to remark that solving \reff{min-tr:qmod}
or \reff{tmp-R:minTr} might not give a representing measure for $y$
whose support has minimum cardinality.
For instance, in Example~\ref{emp:(2,4)6supt},
the tms $y$ there admits a $6$-atomic measure
(with support $\{\pm (1,1), \pm (1,-1),\\ (1,2), (-2,1)\}$,
but we got a $7$-atomic measure by solving \reff{tmp-R:minTr}.

\begin{exm}[random examples]
Now we test the performance of solving \reff{tmp-R:minTr}
for finding flat extensions.
We choose $c_k$ such that $c_k^T w = C_k \bullet M_k(w)$
with $C_k=LL^T$ and $L$ being a Gaussian random matrix.
Select $(n,d)$ from Table~\ref{tab-Ball:(n,d)}.
Like in Example~\ref{rand:Ball}, for each pair $(n,d)$,
randomly generate $N=\binom{n+d}{d}$ points $u_1,\ldots, u_N$
obeying a Gaussian distribution,
and set $y = a_1 [u_1]_d + \cdots a_N [u_N]_d$ ($a_i>0$ is randomly chosen).
For each instance, solve \reff{tmp-R:minTr} starting with $k=d$.
If an optimal $w^*$ of \reff{tmp-R:minTr}
has a flat truncation $w^*|_t(t\geq d)$, we stop; otherwise,
increase $k$ by one, and solve \reff{tmp-R:minTr} again.
Repeat this process until a flat extension of $y$ is found.
In all the tested instances, we were able to find a flat extension of $y$
for some $k \geq d$.
In the computations, we sometimes obtained $|\supp{\mu}| < N$
and sometimes $|\supp{\mu}|>N$, while the former was gotten more often.
\qed
\end{exm}

The above random example leads us to believe that \reff{tmp-R:minTr} is
practical for solving TMPs.
We now give a sufficient condition for
\reff{tmp-R:minTr} to produce a flat extension.

\begin{theorem}
In \reff{tmp-R:minTr}, let $d=2d_0+1$, $k=d_0+1$ and
$c_k^T[x]_{2k} = [x]_k^TC_k[x]_k$ be such that $C_k \succ 0$.
If the tms $y\in \mathscr{M}_{n,d}$ admits a $\rank \,M_{d_0}(y)$-atomic measure,
then the optimizer $w^*$ of \reff{tmp-R:minTr} is flat.
\end{theorem}
\begin{proof}
Let $r = \rank \,M_{d_0}(y)$ and $\mu$ be a $r$-atomic measure admitted by $y$. Set
\[
z = \int [x]_{d+1}d\mu, \quad
M_{d_0+1}(z) = \bbm  M_{d_0}(y) & B \\ B^T & Z \ebm,
\quad C_k = \bbm C_{11} & C_{12} \\ C_{12}^T & C_{22} \ebm.
\]
The matrix $B$ consists of moments $y_\af$ with $|\af|\leq d$. Note that
\[
r = |\supp{\mu}| \geq \rank\,M_{d_0+1}(z) \geq \rank M_{d_0}(y) =r.
\]
So $\rank \,M_{d_0+1}(z)=r$. Then $M_{d_0+1}(z) \succeq 0$ implies
\[
Z = B^T M_{d_0}(y)^+ B \succeq 0,
\]
where the superscript $+$ denotes the Moore-Penrose Pseudo inverse of a matrix.
For every $w$ that is feasible for \reff{tmp-R:minTr}, write
\[
M_{d_0+1}(w) = \bbm  M_{d_0}(y) & B \\ B^T & W \ebm,
\]
\[
C_k \bullet M_{d_0+1}(w) = C_{11} \bullet M_{d_0}(y) + 2 C_{12} \bullet B + C_{22} \bullet W.
\]
Thus, the objective of \reff{tmp-R:minTr} can be equivalently replaced
by $C_{22} \bullet W$. If $M_{d_0+1}(w) \succeq 0$,
then $W-Z\succeq 0$ and
$
C_{22} \bullet W \geq C_{22} \bullet Z.
$
So, $C_{22} \bullet Z$ is a lower bound of $C_{22} \bullet W$
for all feasible $w$. This implies that $z$ is an optimizer of \reff{tmp-R:minTr}.
Indeed, $z$ is the unique optimizer of \reff{tmp-R:minTr}.
Suppose $w$ is another optimizer, then
$
C_{22} \bullet (W-Z) = 0.
$
The fact that $W\succeq Z$ and $C_{22} \succ 0$ implies $W=Z$.
Hence $w=z$. Clearly, $z$ is flat, and the proof is complete.
\end{proof}

\noindent
{\it Remark:}
When $K$ in \reff{def:K} is noncompact,
we can similarly solve the SDP problem:
\be
\left\{
\baray{rl}
\underset{ w\in\mathscr{M}_{n,2k} }{\min} &   c_k^Tw   \\
s.t. & w|_d =y,  M_{t-d_i}(g_i \ast w)  \succeq 0, \, i =0, 1,\ldots,m
\earay \right.
\ee
in analogue to \reff{min-tr:qmod} and \reff{tmp-R:minTr}.
To guarantee it always has a minimizer $w^*$,
we typically choose $c_k$ as $c_k^T[x]_{2k} = [x]_k^T C_k [x]_k$
with $C_k \succ 0$, like in \reff{tmp-R:minTr}.
If a truncation $w^*|_t \,(t\geq d)$ is flat,
we get a flat extension of $y$.

\bigskip \noindent
{\bf Acknowledgement}
The authors thank Larry Fialkow for helpful comments on this work.
J.~William Helton was partially supported by NSF grants
DMS-0700758, DMS-0757212, and the Ford Motor Co.
Jiawang Nie was partially supported by NSF grants
DMS-0757212 and DMS-0844775.

\appendix
\section{Optimization with measures}
\setcounter{equation}{0}

We prove that for a dense set of optimization problems
\reff{min-CM:m(y)}, the optimal solutions are measures having
finite supports. This justifies condition \reff{hyp:fin-sup} in \S \ref{sec:heurThy}.

\begin{lem} \label{extrm=>finite}
Let $y$ be a tms in $\mathscr{M}_{n,d}$ and $K$ be a subset of $\re^n$.
Assume the set $meas(y,K)$ is nonempty.
If a measure $\mu$ is an extreme point of $meas(y,K)$,
then $\mu$ is finitely atomic and $|\supp{\mu}|\leq \binom{n+d}{d}$.
\end{lem}
\begin{proof}
We prove by contradiction. Suppose $\mu$ is extreme and $|\supp{\mu}| > N:=\binom{n+d}{d}$.
Then we can decompose $\supp{\mu}$ in a way such that
\[
\supp{\mu} = \bigcup_{j=1}^{N+1} S_j, \quad S_i \cap S_j = \emptyset,
\quad \mu(S_j) >0, \quad \forall \, i\ne j.
\]
For each $j$, let $\mu_j$ be the restriction of $\mu$ on $S_j$, i.e., $\mu_j = \mu|_{S_j}$.
Note that $\dim \re[x]_d = N$.
So there exists $(a_1,\ldots, a_{N+1}) \ne 0$ satisfying
\[
\sum_{j=1}^{N+1} a_j \int_K x^\af d \mu_j = 0, \quad \forall \, |\af| \leq d.
\]
Define a measure $\nu$ (possibly non-positive) such that
\[
\nu = \nu|_{S_1} + \cdots + \nu|_{S_{N+1}}, \quad
\nu|_{S_j} = a_j \mu_j, \, j=1,\ldots, N+1.
\]
Note that $\supp{\nu} \subseteq \supp{\mu}$.
Then, for $\eps >0$ sufficiently small,
$\mu \pm \eps \nu$ are nonnegative Borel measures supported on $K$,
representing the same tms $y$ and
\[ \mu = (\mu+\eps \nu)/2 + (\mu - \eps \nu)/2. \]
This contradicts the extremality of $\mu$ in $meas(y,K)$.
\end{proof}

\def\cT{\mathcal T}

\begin{theorem}
Let $y$ be a tms in $\mathscr{M}_{n,d}$.
For any real analytic function $f$ on a compact set $K$ and $\eps>0$,
there exists an real analytic function $\tilde f$ on $K$ such that
$\|\tilde{f} - f \|_{\infty} \leq \eps $ and every minimizer of
\be \label{min:tld-f}
\min_\mu  \quad \int_K   \tilde{f} d \mu
\quad \mbox{ s.t. }  \quad \mu \in meas(y,K)
\ee
has cardinality at most $\binom{n+d}{d}$.
Let {\tt Sol} be the set of all optimizers above. Then
\be  \label{card<=nd}
\Big|\bigcup_{ \mu \in {\tt Sol} } \{ \supp{\mu} \} \Big| \leq \binom{n+d}{d}.
\ee
\end{theorem}
\begin{proof}
For a given real analytic function $f$ on $K$, consider
the optimization problem
\be \label{min:int(f)-K}
\min_\mu  \quad \int_K  f d \mu
\quad \mbox{ s.t. }  \quad \mu \in meas(y,K).
\ee
The set of  positive measures on  $K$ whose total masses equal
$y_0$ is compact in the weak-$\ast$ topology (denote this topology by $\cT$)
by the Alaoglu Theorem (cf. \cite[Theorem V.3.1]{Conway}).
Recall that the weak-$\ast$ topology is the topology
on the measures regarded as the dual space of the Banach space $\mathscr{C}(K)$.
It is the weakest topology for which the
convergence of a sequence  of measures $\mu_n \to \mu$ implies
that for every $h \in \mathscr{C}(K)$
\be \label{eq:wkstar}
\int_K  h d \mu_n   \to \int_K  h d \mu.
\ee
This implies that every moment of $\mu_n$ converges to
the corresponding one of $\mu$. Hence, $meas(y,K)$
is $\cT$-closed inside a compact set, and it is also $\cT$-compact.

\def\cE{\mathcal E}

Thus, by compactness of its feasible set,
the optimization problem~\eqref{min:int(f)-K} has a minimizer,
which is a  generalized convex combination of certain extreme points,
by the  Choquet-Bishop-de Leeuw Theorem  (cf. \cite{P66}).
To be more precise, this means that for every $\hat \mu \in meas(y,K)$,
there exists a probability measure $\Gamma$ on the set $\cE$ of extreme points
of $meas(y,K)$ such that
\[ \ell(\hat \mu )= \int_\cE \ell( \mu)\; d\Gamma(\mu)\]
for every affine function $\ell$ on $meas(y,K)$.

Let $\gamma^*$ be the optimal value of \eqref{min:int(f)-K}
and $\tilde{\mu}$ be a minimizer,
and let $\widetilde{\Gamma}$ denote the probability measure on $\cE$
representing $\tilde{\mu}$ and let
$\widetilde{\cE}$ denote the support of $\widetilde{\Gamma}$. Then
$$\gamma^* = \int_K f d \tilde{\mu} = \int_{ \widetilde{\cE} }
\left(\int_K f d\mu \right) d\widetilde{\Gamma} (\mu).$$
By optimality of $\gamma^*$,
$\int_K f d \mu \geq \gamma^*$ for all $\mu \in  \widetilde{\cE}$.
Indeed, $\int_K f d \mu = \gamma^*$ for all $\mu \in  \widetilde{\cE}$.
Otherwise, suppose  $\int_K f d \mu > \gamma^*$ on a set of $\mu$ having positive
$\widetilde{\Gamma}$ measure. Then
$$\gamma^* =  \int_{ \widetilde{\cE} }
\left(\int_K f d\mu \right) d \widetilde{\Gamma}(\mu) >  \gamma^*,$$
which yields a contradiction.
This implies  $\int_K f d \mu = \gamma^*$ on $\widetilde{\cE}$.
Choose a measure $\mu^* \in  \widetilde{\cE}$. It is extreme and
the optimum of \eqref{min:int(f)-K} is attained at $\mu^*$.

By Lemma \ref{extrm=>finite},
we have  $|\supp{\mu^*}|\leq N:=\binom{n+d}{d}$.
Without loss of generality, we can normalize $y_0=\int_K d\mu^*=1$ and denote
$
\supp{\mu^*} = \{u_1, \ldots, u_r\}.
$
Let $e(x)$ be the exponential function defined as
\[
e(x) =  \eps \cdot \exp(-\|x-u_1\|_2^2 \cdots \|x-u_r\|_2^2).
\]
Clearly, $e(x) =\eps$ for all $x\in \supp{\mu^*}$
and $0< e(x) < \eps$ for all $x \not\in \supp{\mu^*}$.
This implies that
$$
\max_{\mu \in meas(y,K) }  \int_K e \; d \mu = \int_K e \; d \mu^* =\eps. $$
\def\tf{\tilde f}
Set $\tf := f- e$ and
\be
\label{eq:minfeps}
\beta:=
\min_{\mu \in meas(y,K) } \int_K \tf \; d \mu.
\ee
Then
\be
\beta \ \geq \
\min_{\mu \in meas(y,K) } \int_K f \; d \mu-
\max_{\mu \in meas(y,K) }  \int_K e \; d \mu =
\int_K f \; d \mu^* - \eps.
\ee
On the other hand
$$
\int_K f \; d \mu^* - \eps =
\int_K  \tf \; d \mu^* \ \geq
\min_{\mu \in meas(y,K) } \int_K  \tf \; d \mu =\beta
$$
because $\mu^*$ belongs to $meas(y,K)$. Thus
\[
\beta = \int_K f \; d \mu^* - \eps
= \int_K \tf d \mu^*
\]
and $\mu^*$ is a minimizer of \eqref{eq:minfeps} as well as \eqref{min:tld-f}.

Suppose $\tilde \mu^*$ is another minimizer of
\eqref{eq:minfeps}.
We wish to show that
\be \label{supptinclu}
suppt({\tilde \mu^*})  \subseteq \{u_1, \ldots, u_r\}.
\ee
If this is not true, then
$\int e \; d {\tilde \mu^*} <\eps$ and
$$
\beta=
\int_K f \; d \mu^*- \eps
\  \ < \ \
  \min_{\mu \in meas(y,K) } \int_K f \; d \mu-
\int_K e \; d{\tilde  \mu}^*.
$$
This implies
$$
\beta <  \int_K f \; d{\tilde  \mu}^*-
\int_K e  \; d{\tilde  \mu}^* = \int_K \tf d{\tilde  \mu}^*  =\beta,
$$
which is a contradiction.

The inequality \reff{card<=nd} follows from the first part,
because the average of all the minimizers
is still a minimizer.
\end{proof}

\noindent {\it Remark:}
As we can see in the above proof,
if $\mu^*$ is a unique minimizer of \reff{min-CM:m(y)},
then $\mu^*$ must have finite support and $|\supp{\mu^*}|\leq \binom{n+d}{d}$.

\medskip

The following lemmas are used in the proof of Theorem~\ref{thm:cvg-flat},

\begin{lem} \label{Ek-Bd:cw-tail}
Suppose $K\subseteq B(0,R)$ with $R<1$ and $c$ is in $\mathscr{M}_{n,\infty}^2$.
If a tms $w$ belongs to $E_k(y) \cup E_\infty(y)$, then
\[ \|w\|_2\leq y_0/(1-R^2), \]
and for any integer $ t >0$ it holds that
\be \label{cw-tail<R2t+2}
\Big| \sum_{ |\af| >2t } c_\af w_\af \Big|
\leq \|c\|_2 \cdot y_0 \cdot R^{2t+2}/(1-R^2).
\ee
\end{lem}
\begin{proof}
First, we consider the case that $w \in E_k(y)$.
The condition $M_{t-1}(\rho \ast w) \succeq 0$ in \reff{min-tr:qmod}
implies that for every $t=1,\ldots,k$
\[
R^2 \mathscr{L}_w(\|x\|_2^{2t-2}) - \mathscr{L}_w(\|x\|_2^{2t}) \geq 0.
\]
A repeated application of the above gives
\[
\mathscr{L}_w(\|x\|_2^{2t}) \leq  R^{2t} y_0.
\]
Since $M_k(w)$ is positive semidefinite, we can see that
\[
\|w\|_2 \leq \|M_k(w)\|_F \leq Trace(M_k(w)) = \sum_{t=0}^k \sum_{|\af| = t}\, w_{2\af},
\]
\[
\sum_{|\af| = t}\, w_{2\af} = \mathscr{L}_w(\sum_{|\af| = t} \,x^{2\af} )
 \leq \mathscr{L}_w(\|x\|_2^{2t}) \leq  R^{2t} y_0.
\]
Thus, it holds that
\[
\|w\|_2 \leq  y_0(1+R^2+\cdots+R^{2k}) \leq y_0/(1-R^2).
\]
When $k>t$, we have
\[
\Big| \sum_{ 2t< |\af| \leq 2k } \, c_\af w_\af \Big|^2 \leq
\Big(\sum_{ 2t< |\af| \leq 2k } \, c_\af^2\Big)
\Big(\sum_{ 2t< |\af| \leq 2k } \, w_\af^2\Big).
\]
Let $M_{t,k}$ be the principle submatrix of $M_k(w)$
whose row and column indices $\bt$ satisfy $ t < |\bt| \leq k$.
Clearly, $M_{t,k} \succeq 0$ and we similarly have
\[
\Big(\sum_{ 2t< |\af| \leq 2k } \, w_\af^2\Big)^{1/2}
\leq  \| M_{t,k} \|_F \leq Trace(M_{t,k})
\]
\[ =
\sum_{j=t+1}^k \sum_{|\af|=j} \, w_{2\af}
\leq \sum_{j=t+1}^k  \,\mathscr{L}_w(\|x\|_2^{2j})
\leq \sum_{j=t+1}^k \, y_0 R^{2j}.
\]
Combining all of the above , we get
\[
\Big| \sum_{ 2t< |\af| \leq 2k } \, c_\af w_\af \Big| \leq
\|c\|_2 \cdot y_0 \cdot (R^{2t+2}+\cdots+R^{2k} )
\leq \|c\|_2 \cdot y_0 \cdot R^{2t+2}/(1-R^2).
\]
Note that $w_\af=0$ if $|\af|>2k$. So, \reff{cw-tail<R2t+2} is true.

Second, we consider the case that $w \in E_\infty(y)$.
The sequence $w$ admits a $K$-measure $\mu$ such that
$w_\af = \int_K x^\af d \mu$ for every $\af$.
For every $k$, consider the truncation $z = \int_K [x]_{2k} d \mu$.
Clearly, the tms $z$ is feasible for \reff{min-tr:qmod}.
By part (i), we have
\[
\|z\|_2\leq y_0/(1-R^2), \qquad
\Big| \sum_{ 2t< |\af| \leq 2k } \, c_\af z_\af \Big|
\leq \|c\|_2 \cdot y_0 \cdot R^{2t+2}/(1-R^2).
\]
The above is true for all $k$, and implies \reff{cw-tail<R2t+2}
by letting $k\to\infty$.
\end{proof}

\begin{lem} \label{optval:cvg}
Assume $K\subseteq B(0,R)$ with $R<1$, and $c$ is in $\mathscr{M}_{n,\infty}^2$.
Let $\gamma_k, \gamma$ be the optimal values of \reff{min-tr:qmod}
and \reff{min-CM:m(y)} respectively. Then, we have the convergence
$\gamma_k \to \gamma$ as $k \to \infty$.
\end{lem}
\begin{proof}
By Lemma~\ref{Ek-Bd:cw-tail},
for an arbitrary $\eps>0$, there exists $t$ such that
\be \label{tail<eps}
\Big| \sum_{ |\af| > t }  c_\af w_\af \Big|  < \eps,
\quad \forall \, w \in E_k(y) \cup E_\infty(y).
\ee
This implies that for every $w \in E_k(y) \cup E_\infty(y)$ with $k\geq t$, it holds that
\be \label{cw-trun<eps}
|\langle c, w \rangle - c_t^Tw|_{2t} | < \eps.
\ee
Now we consider the truncated optimization problems
\be  \label{min-ctw}
\underset{ w\in\mathscr{M}_{n,2k} }{\min} \quad   c_t^Tw|_{2t}
\qquad s.t. \quad   w \in E_k(y)
\ee
and
\be  \label{opt-ct(x)-mu}
\min_{ w }  \quad c_t^Tw|_{2t}
\quad \mbox{ s.t. }  \quad w \in E_\infty(y).
\ee
Let $\tau_k$ and $\tau$ be the optimal values of \reff{min-ctw} and \reff{opt-ct(x)-mu}
respectively. Note that \reff{min-ctw} is
a semidefinite relaxation of \reff{opt-ct(x)-mu}.
Thus, we have $\tau_k \to \tau$, by Theorem~1 of Lasserre \cite{Las08}.
The inequality \reff{cw-trun<eps} implies that for all $k\geq t$
\[
\underset{ w\in E_k(y) }{\min}  c_t^T w|_{2t} - \eps
\leq \underset{ w\in E_k(y) }{\min} c_k^Tw  \leq
\underset{ w\in E_k(y) }{\min}  c_t^T w|_{2t} + \eps.
\]
This shows that $|\gamma_k-\tau_k| \leq \eps$ for $k\geq t$.
Applying a similar argument to \eqref{opt-ct(x)-mu}, we can get $|\gamma-\tau| \leq \eps$.
Note that
\[
| \gamma_k - \gamma |  \leq | \gamma_k - \tau_k |  +
| \tau_k - \tau | + | \tau - \gamma |.
\]
Because $\tau_k \to \tau$, one can get that
\[
0 \leq \limsup_{k\to\infty} | \gamma_k - \gamma | \leq   2\eps.
\]
Since $\eps>0$ is arbitrary, we must have $\gamma_k \to \gamma$
as $k\to \infty$.
\end{proof}

\end{document}